\documentclass[a4paper,11pt,reqno]{amsart}

\usepackage{fullpage}

\usepackage{dsfont,amssymb,amsmath, amsfonts, amsthm, mathtools, color,mathalfa,mathrsfs}
\usepackage[matrix,arrow]{xy}

\usepackage{hyperref} 

\usepackage{tikz}
\usepackage{calc}
\usepackage{ifthen}

\allowdisplaybreaks[4]

\numberwithin{equation}{section}

\newtheorem{theorem}{Theorem}[section]
\newtheorem{corollary}[theorem]{Corollary}
\newtheorem{assumption}[theorem]{Assumption}
\newtheorem{lemma}[theorem]{Lemma}
\newtheorem{proposition}[theorem]{Proposition}
\theoremstyle{remark}
\newtheorem{remark}[theorem]{Remark}
\newtheorem{example}[theorem]{Example}
\theoremstyle{definition}
\newtheorem{definition}[theorem]{Definition}

\newcommand\bp{\begin{proof}}
\newcommand\ep{\end{proof}}

\newcommand\biopencrossl{%
  \mathrel{\scalerel*{>\kern-.4\LMpt\joinrel\blacktriangleleft}{x}}}
\newcommand\biopencrossr{%
  \mathrel{\scalerel*{\blacktriangleright\joinrel\kern-.4\LMpt<}{x}}}

\newcommand\Dhat{{\hat\Delta}}

\newcommand\Ad{\operatorname{Ad}}

\newcommand\HS{\operatorname{HS}}

\newcommand\Op{\operatorname{Op}}
\newcommand\Tr{\operatorname{Tr}}

\newcommand{\C}{{\mathbb C}}
\newcommand{\R}{{\mathbb R}}
\newcommand\T{{\mathbb T}}

\newcommand{\G}{{\mathcal G}}

\newcommand{\N}{{\mathcal N}}
\newcommand{\U}{{\mathcal U}}

\newcommand{\CL}{{\mathcal L}{}}

\newcommand{\CJ}{\mathcal J}

\newcommand{\CH}{\mathcal H}

\newcommand{\CU}{\mathcal U}
\newcommand{\CF}{\mathcal F}

\newcommand{\K}{\mathbb K}

\newcommand{\vf}{\varphi}

\DeclarePairedDelimiter\ang{\langle}{\rangle}

\newlength{\meanderUnitlength}
\newlength{\meanderThickness}
\newcounter{meanderXsep}
\newcounter{meanderYsep}
\newcounter{meanderSum}
\newcounter{meanderTopMax}
\newcounter{meanderBottomMax}
\newcounter{meanderYlevel}
\newcounter{meanderXsize}
\newcounter{meanderYsize}
\newcounter{meanderBulletSize}
\newcounter{meanderSnode}
\newcounter{meanderArcNumber}
\newcounter{meanderloopcount}
\newcounter{meanderA}
\newcounter{meanderB}
\newcounter{meanderM}
\newcounter{meanderH}
\newcounter{meanderarcstyle}
\newcommand{\resetmeanderdefaults}{%
\setlength{\meanderUnitlength}{\unitlength} 
\setlength{\meanderThickness}{0.6pt}
\setcounter{meanderXsep}{24}
\setcounter{meanderYsep}{10}
\setcounter{meanderBulletSize}{3}
\setcounter{meanderarcstyle}{0}
}
\resetmeanderdefaults
\newcommand{\meander}[2]{%
\setcounter{meanderTopMax}{0}
\setcounter{meanderBottomMax}{0}
\setcounter{meanderSum}{0}
\foreach \x in {#1}
	{
	\addtocounter{meanderSum}{\x}
	\ifthenelse{\x > \themeanderTopMax}{\setcounter{meanderTopMax}{\x}}{}
	}
\foreach \x in {#2}
	{
	\ifthenelse{\x > \themeanderBottomMax}{\setcounter{meanderBottomMax}{\x}}{}
	}
\setcounter{meanderYlevel}{\themeanderYsep*\themeanderBottomMax}
\setcounter{meanderXsize}{\themeanderXsep*\themeanderSum-\themeanderXsep}
\setcounter{meanderYsize}{\themeanderYsep*(\themeanderTopMax +\themeanderBottomMax)}
\setlength{\unitlength}{\meanderUnitlength}
\begin{array}{c}
\begin{picture}(\themeanderXsize,\themeanderYsize)
\linethickness{\meanderThickness}
\ifthenelse{\themeanderBulletSize>0}{\multiput(0,\themeanderYlevel)(\themeanderXsep,0){\themeanderSum}{\circle*{\themeanderBulletSize}}}{}
\setcounter{meanderSnode}{0}
\foreach \x in {#1}
	{
	\addtocounter{meanderSnode}{1}
	\ifthenelse{\isodd{\x}}{\setcounter{meanderArcNumber}{(\x-1)/2}}{\setcounter{meanderArcNumber}{\x/2}}
	\setcounter{meanderloopcount}{0}
	\whiledo{\themeanderloopcount < \themeanderArcNumber}
		{
		\setcounter{meanderA}{(\themeanderSnode+\themeanderloopcount-1)*\themeanderXsep}
		\setcounter{meanderB}{(\themeanderSnode+\x - \themeanderloopcount-2)*\themeanderXsep}
		\setcounter{meanderM}{(\themeanderA+\themeanderB)/2}
		\setcounter{meanderH}{\themeanderYsep*(\x-2*\themeanderloopcount-1)+\themeanderYlevel}
		\ifthenelse{\themeanderarcstyle=1}%
			{
			\qbezier(\themeanderA,\themeanderYlevel)(\themeanderM,\themeanderH)(\themeanderM,\themeanderH)
			\qbezier(\themeanderM,\themeanderH)(\themeanderB,\themeanderYlevel)(\themeanderB,\themeanderYlevel)
			}
			{
			\qbezier(\themeanderA,\themeanderYlevel)(\themeanderA,\themeanderH)(\themeanderM,\themeanderH)
			\qbezier(\themeanderM,\themeanderH)(\themeanderB,\themeanderH)(\themeanderB,\themeanderYlevel)
			}
		\addtocounter{meanderloopcount}{1}
		}
	\addtocounter{meanderSnode}{\x-1}
	}
\setcounter{meanderSnode}{0}
\foreach \x in {#2}
	{
	\addtocounter{meanderSnode}{1}
	\ifthenelse{\isodd{\x}}{\setcounter{meanderArcNumber}{(\x-1)/2}}{\setcounter{meanderArcNumber}{\x/2}}
	\setcounter{meanderloopcount}{0}
	\whiledo{\themeanderloopcount < \themeanderArcNumber}
		{
		\setcounter{meanderA}{(\themeanderSnode+\themeanderloopcount-1)*\themeanderXsep}
		\setcounter{meanderB}{(\themeanderSnode+\x - \themeanderloopcount-2)*\themeanderXsep}
		\setcounter{meanderM}{(\themeanderA+\themeanderB)/2}
		\setcounter{meanderH}{-\themeanderYsep*(\x-2*\themeanderloopcount-1)+\themeanderYlevel}
		\ifthenelse{\themeanderarcstyle=1}%
			{
			\qbezier(\themeanderA,\themeanderYlevel)(\themeanderM,\themeanderH)(\themeanderM,\themeanderH)
			\qbezier(\themeanderM,\themeanderH)(\themeanderB,\themeanderYlevel)(\themeanderB,\themeanderYlevel)
			}
			{
			\qbezier(\themeanderA,\themeanderYlevel)(\themeanderA,\themeanderH)(\themeanderM,\themeanderH)
			\qbezier(\themeanderM,\themeanderH)(\themeanderB,\themeanderH)(\themeanderB,\themeanderYlevel)
			}
		\addtocounter{meanderloopcount}{1}
		}
	\addtocounter{meanderSnode}{\x-1}
	}
\end{picture} \\
\end{array}
}
\newcommand{\meanderstyle}[1]{%
\ifthenelse{\equal{#1}{plain}}{\setcounter{meanderBulletSize}{0}\setcounter{meanderarcstyle}{0}}{}
\ifthenelse{\equal{#1}{default}}{\setcounter{meanderBulletSize}{3}\setcounter{meanderarcstyle}{0}}{}
\ifthenelse{\equal{#1}{line}}{\setcounter{meanderBulletSize}{3}\setcounter{meanderarcstyle}{1}}{}
\ifthenelse{\equal{#1}{plainline}}{\setcounter{meanderBulletSize}{0}\setcounter{meanderarcstyle}{1}}{}
}
\newcommand{\meandersize}[1]{%
\ifthenelse{\equal{#1}{Huge}}{\setlength{\meanderUnitlength}{1.8pt}}{}
\ifthenelse{\equal{#1}{huge}}{\setlength{\meanderUnitlength}{1.6pt}}{}
\ifthenelse{\equal{#1}{Large}}{\setlength{\meanderUnitlength}{1.4pt}}{}
\ifthenelse{\equal{#1}{large}}{\setlength{\meanderUnitlength}{1.2pt}}{}
\ifthenelse{\equal{#1}{default}}{\setlength{\meanderUnitlength}{1.0pt}}{}
\ifthenelse{\equal{#1}{Small}}{\setlength{\meanderUnitlength}{0.8pt}}{}
\ifthenelse{\equal{#1}{small}}{\setlength{\meanderUnitlength}{0.6pt}}{}
\ifthenelse{\equal{#1}{Tiny}}{\setlength{\meanderUnitlength}{0.4pt}}{}
\ifthenelse{\equal{#1}{tiny}}{\setlength{\meanderUnitlength}{0.2pt}}{}
\ifthenelse{\equal{#1}{minuscule}}{\setlength{\meanderUnitlength}{0.1pt}}{}
}
\newcommand{\meanderthickness}[1]{%
\ifthenelse{\equal{#1}{Thin}}{\setlength{\meanderThickness}{0.2pt}}{}
\ifthenelse{\equal{#1}{thin}}{\setlength{\meanderThickness}{0.4pt}}{}
\ifthenelse{\equal{#1}{default}}{\setlength{\meanderThickness}{0.6pt}}{}
\ifthenelse{\equal{#1}{thick}}{\setlength{\meanderThickness}{0.8pt}}{}
\ifthenelse{\equal{#1}{Thick}}{\setlength{\meanderThickness}{1pt}}{}
\ifthenelse{\equal{#1}{wide}}{\setlength{\meanderThickness}{1.2pt}}{}
\ifthenelse{\equal{#1}{Wide}}{\setlength{\meanderThickness}{1.4pt}}{}
}

\begin{document}

\title{From projective representations to pentagonal cohomology via quantization}

\author[Gayral]{Victor Gayral}
\email{victor.gayral@univ-reims.fr}

\address{Laboratoire de Math\'ematiques, CNRS FRE 2011, Universit\'e de Reims Champagne-Ardenne,
Moulin de la Housse - BP 1039,
51687 Reims, France}

\author[Marie]{Valentin Marie}
\email{valentin.marie@univ-reims.fr}

\address{Laboratoire de Math\'ematiques, CNRS FRE 2011, Universit\'e de Reims Champagne-Ardenne,
Moulin de la Housse - BP 1039,
51687 Reims, France}

\maketitle
\begin{abstract}
Given a locally compact group $G=Q\ltimes V$ such that $V$ is Abelian and such that the  action
of $Q$ on the Pontryagin dual $\hat V$ has a free orbit of full measure, we construct a family
of unitary dual $2$-cocycles  $\Omega_\omega$ (aka non-formal Drinfel'd twists) whose equivalence classes 
$[\Omega_\omega]\in H^2(\hat G,\mathbb T)$ are 
parametrized by cohomology classes
$[\omega]\in H^2(Q,\mathbb T)$. We prove that the associated locally compact quantum groups are 
isomorphic to  cocycle bicrossed product quantum groups associated to a pair of subgroups of 
the dual semidirect product $Q\ltimes\hat V$, both isomorphic to $Q$, and to a pentagonal cocycle $\Theta_\omega$
explicitly given in terms of the group cocycle \nolinebreak[4]$\omega$.

\end{abstract}

\tableofcontents

\section{Introduction}

Although the theory of locally compact quantum groups in the von Neumann algebraic setting \cite{KV1,KV2} has, to a great extent, reached now
a high level of maturity, the  important question of the construction of examples which are neither compact nor discrete is still widely open. 
The main obstacle being that contrarily to the compact case, the algebraic level of Lie bialgebra quantization 
for noncompact 
Lie groups is not enough in the analytic setting. This has been demonstrated some time ago, for instance 
in the work of Koelink and Kustermans \cite{KK} and of Woronowicz \cite{W1,W2}.

Major progress in this direction has been made by De Commer \cite{DC} by putting quantization of locally compact groups into the
analytic framework of quantum groups. More precisely, De Commer  solved positively the   long standing problem
of existence of left- and right-invariant weights for the deformed von Neumann bialgebra $(W^*(G),\Omega\hat\Delta(.)\Omega^*)$.
Here, $G$ is a   locally compact group
and $\Omega\in W^*(G\times G)$ is a unitary convolution operator    satisfying the $2$-cocycle equation
$$
(\Omega\otimes1)(\Dhat\otimes\iota)(\Omega)=(1\otimes\Omega)(\iota\otimes\Dhat)(\Omega).
$$
Such an operator $\Omega$ is called a \emph{dual unitary  $2$-cocycle}, by analogy with  group $2$-cocycles,
which can be defined as  unitary elements of $L^\infty(G\times G)$ satisfying 
$(\omega\otimes1)(\Delta\otimes\iota)(\omega)=(1\otimes\omega)(\iota\otimes\Delta)(\omega)$.

De Commer's construction of  invariant weights for $(W^*(G),\Omega\hat\Delta(.)\Omega^*)$ 
 uses, as a key tool, the  structure  of \emph{$G$-Galois object}, 
an appropriate notion of free and ergodic action of the group $G$
 on a von Neumann algebra, which can be  attached to a dual unitary  $2$-cocycle.
 
However, constructing a dual unitary  $2$-cocycle for a given locally compact group is already a nontrivial task.
This is, in the first place, because of emergence   of the representation theory of $G$ in the picture. For instance, in \cite{BGNT3} it is shown that if 
 the group von Neumann algebra $W^*(G)$ is a type I factor (equivalently, if $G$ possesses a unique class of irreducible 
 and square-integrable representations $[\pi]$), then there exists a (canonical class of) dual unitary  $2$-cocycle on $G$. But  to
 obtain a concrete representative of this class, we need  an explicit unitary equivalence  between the left-regular representation $\lambda$
 and the tensor product representation $\pi\otimes\pi^c$ with the contragredient $\pi^c$. Such an intertwiner is called a \emph{unitary equivariant 
 quantization} of $G$, because it transforms square-integrable functions on $G$ to Hilbert-Schmidt operators on the representation space of $\pi$,
 in a $G$-equivariant way. More generally, let $\pi$ be a \emph{projective representation} of $G$ (always assumed to be
  irreducible and square-integrable) with associated
  $2$-cocycle $\omega\in Z^2(Q,\mathbb T)$.  
  It is also proven in \cite{BGNT3} that if the twisted group von Neumann algebra $W^*(G,\omega)$ is a type I factor and if 
 a unitary equivariant quantization exists, then a dual unitary  $2$-cocycle on $G$  does exist too.  In both cases (projective or genuine representation),
  the main  remaining task is to construct such a quantization map.
  
  For a semidirect product $G=Q\ltimes V$, with $V$ Abelian and with $Q$ possessing a free dual orbit of full measure, it is not
  hard to see that the group von Neumann
  algebra $W^*(G)$ is indeed a type I factor.  One important result of  \cite{BGNT3} is that the Kohn-Nirenberg quantization of the self-dual Abelian
  group $\hat V\times V$ can be used to construct a unitary equivariant quantization of $G$, leading therefore to an explicit dual unitary  $2$-cocycle $\Omega$.
   Moreover,  the deformed quantum group $(W^*(G),\Omega\hat\Delta(.)\Omega^*)$ is isomorphic to a \emph{bicrossed product}
  quantum group \cite{BS1,BSV,VV}, based on a matched pair of subgroups of the dual semidirect product $Q\ltimes \hat V$, both isomorphic to
  $Q$.
  
   In the present paper, we show that if $\omega$ is any nontrivial $2$-cocycle on $Q$, seen in a natural way as a nontrivial $2$-cocycle on $G$, then the 
  twisted group von Neumann algebra $W^*(G,\omega)$ is again a type I factor. Mimicking the  Kohn-Nirenberg quantization from a 
  functional approach, we construct a unitary equivariant quantization of $G$ and compute the associated dual unitary  $2$-cocycle $\Omega_\omega$.
  Our main result is that the deformed quantum group $(W^*(G),\Omega_\omega\hat\Delta(.)\Omega_\omega^*)$ is now isomorphic to a
   \emph{cocycle bicrossed product}
  quantum group \cite{VV}, associated with the same  matched pair and with a nontrivial pentagonal $2$-cocycle  $\Theta_\omega$ (see below for the definitions).
  Moreover, we show that the map $\omega\mapsto\Theta_\omega$ induces a group homomorphism from measurable group cohomology 
  to pentagonal cohomology and that the restriction of this map to continuous group cohomology is  injective.

\section{Generalities}

\subsection{Notations}

We use all the conventions and notations of  \cite{BGNT3}, with the single exception 
that here inner products of complex Hilbert spaces are linear on the right.
Let us recall the most importants.
Given a locally compact group $G$ (always assumed to be second countable), we let $L^p(G)$, $p\in[1,\infty]$, be the $L^p$-space associated with a 
left-invariant Haar measure  $dg$. 
Let ${\rm Aut}(G)$ be the group of continuous automorphisms of $G$. 
 The modulus function $|.|_G:{\rm Aut}(G)\to\R_+^*$ and the  modular function $\Delta_G:G\to\R_+^*$ 
are defined by the relations:
\begin{align*}
 \int_Gf(\vf(g))dg&= |\vf|_G^{-1}\int_G f(g)dg,\quad \forall f \in C_c(G),\; \forall \vf\in {\rm Aut}(G),\\
\int_Gf(gh)dg&= \Delta_G(h)^{-1}\int_G f(g)dg,\quad  \forall f\in C_c(G),\;\forall  h\in G.
\end{align*}
For $f\in L^\infty(G)$, let $\check f\in L^\infty(G)$ be defined by
$$
\check f(g):= f(g^{-1}).
$$
The unitary left- and right-regular representations of $G$ on $L^2(G)$ are respectively defined by
$$
(\lambda_gf)(h):=f(g^{-1}h),\quad (\rho_gf)(h):=\Delta_G(g)^{1/2}\,f(hg)\,,\qquad \forall g,h\in G,\quad \forall f\in L^2(G).
$$
The modular conjugations $J$ of $L^\infty(G)$ and $\hat J$ of $W^*(G):=\lambda(G)''$ are given by
$$
Jf:=\bar f, \quad \hat J f=\Delta_G^{-1/2} \bar{\check f},\quad \forall f\in L^2(G).
$$
We will also consider the unitary operator
$$
\CJ:=\hat J J.
$$
The multiplicative unitaries $W,\hat W:L^2(G\times G)\to L^2(G\times G)$ of $G$ and of its dual  are defined by
$$
(Wf)(g,h):=f(g,g^{-1}h),\quad (\hat Wf)(g,h):=f(hg,h).
$$
Here, we will be interested in semidirect products
$G=Q\ltimes V$, where $V$ is Abelian and $Q$ is a closed subgroup of 
 ${\rm Aut}(V)$. Denoting the action of  ${\rm Aut}(V)$ on $V$ by juxtaposition, 
 the group law of $Q\ltimes V$ is
$$
(q,v)(q',v')=(qq',v+qv'),\quad \forall q,q'\in Q,\;\forall v,v'\in V.
$$
In this case, we have the following formulas for a left-invariant Haar measure and for the modular function on $G$:
$$
d(q,v)=\frac{dqdv}{|q|_V},\quad \Delta_G(q,v)=\frac{\Delta_Q(q)}{|q|_V}.
$$
Let $\hat V$ be the Pontryagin dual of $V$. In order to use additive notations both on $V$ and  $\hat V$, it is convenient 
to denote the duality pairing $\hat V\times V\to \T$ by $e^{i\langle \xi,v\rangle}$, $\xi\in \hat V$, $v\in V$. We do not claim
that there is an exponential function here, it is just a notation. We also set $e^{-i\langle \xi,v\rangle}:=\overline{e^{i\langle \xi,v\rangle}}=e^{i\langle -\xi,v\rangle}
=e^{i\langle \xi,-v\rangle}$. The Haar measure $d\xi$ of $\hat V$ is normalized such that the Fourier transform $\CF_V$:
$$
(\CF_Vf)(\xi):=\int_V e^{-i\langle \xi,v\rangle}\,f(v)\,dv,
$$
becomes unitary from $L^2(V)$ to $L^2(\hat V)$. With this convention, we have the relation 
$$
(\CF_Vf)(\xi)=(\CF_{\hat V}^*f)(-\xi)\,,\quad \forall f\in L^2(V),\quad \forall \xi\in\hat V.
$$ 
We also denote by $\CF_{\hat V}:L^2(G)\to L^2(Q\times \hat V, |q|^{-1}dqd\xi)$ the (unitary) partial Fourier transform.
The dual action of $q\in Q$ on $\xi\in \hat V$, 
denoted by $q^\flat\xi$, is defined by 
$$
e^{i\langle q^\flat \xi,v\rangle}=e^{i\langle \xi,q^{-1}v\rangle}\,,\quad \forall v\in V,\;\forall \xi\in\hat V\,,\forall  q\in Q.
$$
 We also have the following relation between the modulus
functions of $V$ and of $\hat V$: 
$$
|q^\flat|_{\hat V}=|q|_V^{-1},\quad\forall q\in Q.
$$
We will always identify a function $f\in L^\infty(G)$ with the bounded operator of multiplication by this function on $L^2(G)$.

\subsection{Galois objects}
\label{GO}
Originating from Hopf algebras theory, the notion  of Galois object has been  adapted by De Commer \cite{DC} to the operator algebraic setting of 
locally compact quantum groups \cite{KV1,KV2}.
The most important thing to know is that given a  Galois object for a locally compact quantum group, one can construct another one 
 (by reflexion across the Galois object).   
Galois objects for   non-Abelian groups, the only case  considered here,  are already very important since they 
allow to construct 
genuine locally compact quantum groups (i.e$.$ not commutative  nor cocommutative).

Let $G$ be a  locally compact group. A \emph{$G$-Galois object} is a pair $(\N,\beta)$
where $\N$ is a von Neumann algebra and  $\beta:G\to{\rm Aut}(\N) $   is an action satisfying  the following three properties.

\begin{enumerate}
\item \emph{Ergodicity.}  The  invariant vectors  are reduced to the scalars:
$$
\{a\in\N:\beta_g(a)=a,\;\forall g\in G\}=\C\,1_\N.
$$

\item  \emph{Integrability.} The canonical  weight $\vf:\N_+\to [0,+\infty]$, defined by
\begin{align}
\int_G \beta_g(a)dg =\vf(a) \,1_\N,\quad a\in\N_+,
\label{vf-def}
\end{align}
is semi-finite.\\

\item  \emph{Unitarity.} Denoting by $\Lambda: {\mathfrak N}_{\vf}=\{a\in\N:\vf(a^*a)<\infty\}\to
 L^2(\N,\vf)$ the GNS-map of the weight $\vf$,
the isometric Galois map
\begin{align}
\label{GM-def}
{\mathcal G}\colon L^2(\N,\varphi)\otimes L^2(\N,\varphi)\to L^2(G;L^2(\N,\varphi)),\ \ {\mathcal G}\big(\Lambda(a)\otimes \Lambda(b)\big)(g)=
\Lambda(\beta_g(a)b),
\end{align}
is surjective.
\end{enumerate}

The simplest type of $G$-Galois objects occurs when $\N\simeq B(\CH)$, for  $\CH$ a separable Hilbert space.
We  then talk about  \emph{$I$-factorial $G$-Galois objects} and we denote by  ${\rm Gal}_I(G)$ the set (see below) of isomorphism classes of such Galois objects.
 In this situation,
one  necessarily has
 $\beta=\Ad\pi$ where $\pi$ is a unitary, irreducible,  square-integrable and projective
 representation of $G$ on $\CH$. 
Rather than a group homomorphism $G\to PU(\CH)$, 
 here we prefer to see a projective representation directly  as a  Borel map
 $\pi:G\to U(\CH)$ satisfying $\pi(g_1)\pi(g_2)=\omega_\pi(g_1,g_2)\pi(g_1g_2)$, where
 $\omega_\pi$ is a $\T$-valued Borel $2$-cocycle on $G$.
  Note that ergodicity and integrability of the action $\Ad\pi$ are respectively
equivalent  to irreducibility and to square-integrability of the projective representation $\pi$ (see \cite[Remark 2.2]{BGNT3}).
Moreover it follows from \cite[Theorem 2.4]{BGNT3} that the Galois map \eqref{GM-def}
is unitary if and only if  the twisted group von Neumann algebra $W^*(G,\omega_\pi)$ is a type $I$ factor.
More precisely, defining $H^2_{I}(G,\mathbb T)$ to be the set  of   classes $[\omega]\in H^2(G,\mathbb T)$ (in measurable cohomology \cite{Moore}) such that  $W^*(G,\omega)$
is a type $I$ factor, it follows by \cite[Theorem 2.4]{BGNT3} that we have an  isomorphism
\begin{align}
\label{Isom1}
{\rm Gal}_I(G)\;\longrightarrow \;H_I^2(G,\mathbb T),\quad [(B(\CH),\Ad\pi)]\mapsto [\omega_\pi].
\end{align}
Note that  $H^2_{I}(G,\mathbb T)$ needs not to be a subgroup of $H^2(G,\mathbb T)$. Indeed, in the example studied by Jondreville   \cite{J}, we have
 $[1]\not\in H_I^2(G,\mathbb T)\ne\emptyset$. In the present article, we will describe a class of examples where $H^2_{I}(G,\mathbb T)$ contains a subgroup of 
 $H^2(G,\mathbb T)$.

For a $I$-factorial $G$-Galois object $(B(\CH),\Ad\pi)$, the  canonical weight \eqref{vf-def} 
is related to the  Duflo--Moore formal degree operator $D$  of the projective representation $\pi$
 by the formula:
\begin{align}
\label{vf-DM}
\varphi=\Tr(D^{1/2}\cdot D^{1/2}).
\end{align}
This follows because  $D$  is the unique positive, possibly unbounded, nonsingular operator  on
$\CH$  satisfying
$$
\int_G|\langle \zeta,\pi(g)\xi\rangle|^2dg=\|D^{1/2}\xi\|^2\|\zeta\|^2\ \ \text{for all}\ \ \xi\in{\rm Dom}(D^{1/2})\ \ \text{and}\ \ \zeta\in \CH.
$$
See \cite{DM} for genuine representations and \cite{A} for projective representations.

\subsection{Dual cocycles}
In this article, we are  interested in the situation where the 
locally compact quantum group associated with
a $I$-factorial $G$-Galois object $(B(\CH),\Ad\pi)$
can also be described as a dual cocycle deformation of $G$. 
Recall that a \emph{dual  unitary $2$-cocycle}  (aka a non-formal Drinfeld twist) is a unitary element $\Omega\in W^*(G)\bar\otimes W^*(G)$ 
    satisfying the relation:
\begin{equation}
\label{eq:dual-cocycle}
(\Omega\otimes1)(\Dhat\otimes\iota)(\Omega)=(1\otimes\Omega)(\iota\otimes\Dhat)(\Omega).
\end{equation}
The cocycle relation ensures coassociativity  of the deformed coproduct $\hat\Delta_\Omega:=\Omega \hat\Delta(.)\Omega^*$ on $W^*(G)$.
The question of existence 
 of invariant weights for $\hat\Delta_\Omega$ has been solved positively by De Commer
\cite{DC} using  Galois objects. 
The   $G$-Galois object associated to a dual  unitary $2$-cocycle $\Omega$ can be described as follows (see \cite[section 5]{DC} and \cite[section 4]{NT}).
Let $\star_\Omega$ be the associative product on the Fourier algebra $A(G)$ (identified as usual with the predual of $W^*(G)$) given by:
\begin{align}
\label{starOmega}
f_1\star_\Omega f_2(g):=(f_1\otimes f_2)\big(\hat\Delta(\lambda_g)\Omega^*\big),\qquad \forall g\in G.
\end{align}
Consider the representation of the  Banach algebra $(A(G),\star_\Omega)$ on $L^2(G)$  given by:
\begin{align}
\label{piOmega}
\pi _\Omega:A(G)\to B(L^2(G)),\quad f\mapsto (f\otimes\iota)(\hat W\Omega^*).
\end{align}
This representation intertwines the adjoint of the right-regular representation $\Ad\rho$, with the left-regular representation $\lambda$:
$$
\Ad\rho_g \big(\pi _\Omega(f)\big)=\pi _\Omega(\lambda_g f),\quad\forall f\in A(G),\quad\forall g\in G.
$$
Set $W^*(\hat G, \Omega):=\pi_\Omega(A(G))''$. This notation  justifies itself because for the trivial cocycle $\Omega=1$, we have $\pi _1(f)=\check f$
and thus $W^*(\hat G,1)=L^\infty(G)$.
 The $G$-Galois object associated with a dual  $2$-cocycle $\Omega$ is then given by 
$(W^*(\hat G, \Omega),\Ad\rho)$.
The canonical weight of this Galois object  reads:
\begin{align}
\label{TVF}
\vf\big(\pi _\Omega(f)^*\pi _\Omega(f)\big)=\|\check f\|_2^2\quad \mbox{for all} \quad f\in A(G) \quad \mbox{such that}\quad \check f\in L^2(G),
\end{align}
and the GNS space of  $\vf$ can be identified with $L^2(G)$ with associated GNS map
$\Lambda:\mathfrak N_{\vf}\to L^2(G)$  uniquely determined by:
\begin{align}
\label{TL}
\Lambda(\pi _\Omega(f))=\check f\quad \mbox{for all} \quad f\in A(G) \quad \mbox{such that}\quad \check f\in L^2(G).
\end{align}
It follows from \cite[Proposition 5.1]{DC} that under the identification $L^2(\N,\vf)\simeq L^2(G)$, the isometric Galois map 
of the $G$-Galois object $(W^*(\hat G, \Omega),\Ad\rho)$
is given by
$$
\mathcal G=\hat W\Omega^*.
$$
 Therefore the Galois map is unitary. Another very important formula for us, proven in \cite[Proposition 5.4]{DC}, concerns the multiplicative 
 unitary of the deformed quantum group
 $(W^*(G),\hat\Delta_\Omega)$:
\begin{align}
\label{DMU}
\hat W_{\Omega}
=(\tilde J\otimes\hat J)\,\Omega\,\hat W^*\,(J\otimes\hat J)\,\Omega^*.
\end{align} 
In this formula, $\tilde J$ denotes the modular involution of the canonical weight \eqref{TVF}. 

Recall that two dual unitary $2$-cocycles $\Omega$ and $\Omega'$ are said to be cohomologous if there exists a unitary $u\in W^*(G)$
such that $\Omega'=(u\otimes u)\Omega\hat\Delta(u)^*$. The set of equivalence classes of dual $2$-cocycles is denoted by
$H^2(\hat G,\mathbb T)$ and the subset of classes $[\Omega]$  such that
$W^*(\hat G, \Omega)$ is a type $I$ factor is denoted by $H^2_I(\hat G,\mathbb T)$. Hence, if $[\Omega]\in H^2_I(\hat G,\mathbb T)$,  
there exists  a unitary irreducible and square-integrable projective representation $\pi$ on $\CH$, such that
under the identification $W^*(\hat G,\Omega)\simeq B(\CH)$,
the action $\Ad\rho$ becomes $\Ad\pi$. Let also
 $\omega_\Omega\in Z^2(G,\mathbb T)$ be the group $2$-cocycle associated with this projective representation.
 Since cohomologous dual $2$-cocycles yield unitarily equivalent $G$-Galois objects, the isomorphism \eqref{Isom1} yields an embedding: 
\begin{align}
\label{Emb1}
H^2_I(\hat G,\mathbb T)\longrightarrow H^2_I(G,\mathbb T),\quad [\Omega] \mapsto [\omega_\Omega].
\end{align}

\subsection{Equivariant quantizations}
\label{EQ}
Our goal  is to go in the direction  opposite  to the one of the embedding \eqref{Emb1}. Namely,   given a projective representation $\pi$ such that
the twisted group von Neumann algebra $W^*(G,\omega_\pi)$ is a type $I$ factor, we aim to construct a dual $2$-cocycle $\Omega$ such that the 
locally compact quantum group associated with the $G$-Galois object $(B(\CH),\Ad\pi)$ is isomorphic to the dual cocycle
deformation  $(W^*(G),\hat\Delta_\Omega)$.

By  \cite[Proposition 2.9]{BGNT3},  we know that this situation occurs precisely  when the left regular representation $\lambda$ is unitarily
equivalent to $\pi\otimes\pi^c$ or,  equivalently, when there exists a unitary map $\Op:L^2(G)\to {\HS}(\CH)$,
 \emph{that we call a unitary equivariant quantization map},
such that
$$
\Ad\pi(g)\big(\Op(f)\big) =\Op(\lambda_gf),\quad\forall f\in L^2(G),\quad\forall g\in G.
$$
Transporting everything with  $\Op$, we can take $L^2(G)$ for  the GNS-space of the canonical weight  $\varphi$
of the $G$-Galois object $(B(\CH),{\rm Ad} \pi)$. With $D$ the Duflo-Moore operator of $\pi$,
 the associated  GNS map $\tilde\Lambda:\mathfrak N_\vf\to L^2(G)$ is  uniquely determined by
\begin{equation}\label{eq:GNS}
\tilde\Lambda(\Op(f)D^{-1/2}):=f\ \ \text{for}\ \ f\in L^2(G)\ \ \text{such that}\ \  \Op(f)D^{-1/2}\in B(\CH),
\end{equation}
and the corresponding Galois map $\tilde\G\colon L^2(G\times G)\to L^2(G\times G)$ reads:
\begin{equation}\label{eq:Galois}
\tilde\G(f_1\otimes f_2)(g,h)=\tilde\Lambda\big(\big(\Ad\pi(g)\big)(\Op(f_1)D^{-1/2})\Op(f_2)D^{-1/2}\big)(h).
\end{equation}
Setting $\CJ:=J\hat J$, we know by  \cite[Proposition 2.9]{BGNT3} that the following unitary operator on $L^2(G)$ defines a dual unitary   $2$-cocycle on $G$:
\begin{equation}
\label{eq:Omega}
\Omega:=(\CJ\otimes \CJ)\tilde\G^*(1\otimes \CJ)\hat W,
\end{equation}
and that the $G$-Galois objects $(B(\CH),\Ad\pi)$ and  $(W^*(\hat G,\Omega),\Ad\rho)$ are indeed isomorphic.

Observe that $H^2_I(G,\mathbb T)$ may or may not contain the trivial class $[1]\in H^2(G,\mathbb T)$.
This happens   exactly when $W^*(G)$ is a type $I$ factor (see \cite[Theorem 2.13]{BGNT3}).
In this case, there is a unique  (class of)
irreducible and square-integrable   (genuine) representation $\pi$ and   a unitary equivariant quantization map always exists.
But the   proof of  existence of the quantization map is not  constructive at all 
and  the question of the explicit construction of the dual $2$-cocycle \eqref{eq:Omega} 
remains an open problem. This is what had been undertaken in \cite{BGNT3}
for a class of semidirect products. 

Moreover,  associated to a nontrivial class in $ H^2_I(G,\mathbb T)$, 
we don't even know if a unitary equivariant quantization map always exists. The main objective of the present
article is to construct such quantization maps  for projective representations of the class of semidirect products
considered in \cite{BGNT3}.

Assume that we are given a projective representation $\pi$ such that $[\omega_\pi]\in H^2_I(G,\mathbb T)$ together with a unitary equivariant quantization map
$\Op:L^2(G)\to\HS(\CH)$. 
An important step for us is to determine an explicit isomorphism between the $G$-Galois objets $(B(\CH),\Ad\pi)$ and  $(W^*(\hat G,\Omega),\Ad\rho)$
and to this aim it is convenient to consider a third equivalent $G$-Galois object.  So, define the associative product $\star$ on $L^2(G)$ 
by transporting the product of Hilbert-Schmidt operators $\HS(\CH)$ via the quantization map:
\begin{align}
\label{star-prod}
f_1\star f_2:=\Op^*\big(\!\Op(f_1)\,\Op(f_2)\big).
\end{align}
For $f\in L^2(G)$, we denote by $L^\star(f)$ the bounded operator on $L^2(G)$ of left-$\star$-multiplication:
$$
L^\star(f_1)f_2:=f_1\star f_2.
$$
Then   $(L^\star(L^2(G))'',\Ad\lambda)$ defines  a $G$-Galois object canonically isomorphic to $(B(\CH),\Ad\pi)$. Explicitly, the isomorphism
reads $L^\star(f)\mapsto \Op(f)$, whenever $f\in L^2(G)$. As observed in \cite[Remark 2.11]{BGNT3}, the product  $\star$ does not necessarily
coincide on $A(G)\cap L^2(G)$ with the product $\star_\Omega$ defined in \eqref{starOmega}. 
But when they do coincide, then the representation $\pi_\Omega$ of the Banach algebra
$(A(G),\star_\Omega)$ given in \eqref{piOmega} is related  to the representation $L^\star$ of the Hilbert algebra $(L^2(G),\star)$. Namely,
 for $f\in A(G)\cap L^2(G)$    we have the equality (on some
domain--see the discussion at the beginning of Section 3.3 in \cite{BGNT3}):
$$
\pi_\Omega(f)=\CJ \Delta_G^{-1/2} L^\star(f)\Delta_G^{1/2} \CJ.
$$
Therefore, to find the explicit isomorphism between  the $G$-Galois objects $(B(\CH),\Ad\pi)$ and  $(W^*(\hat G,\Omega),\Ad\rho)$, 
it suffices to find a connection between the operators $L^\star (f)$ and $\Delta_G^{\pm 1/2}$.

\subsection{Cocycle bicrossed products and pentagonal cohomology}
\label{CBPPC}

Let $(G_1,G_2; G)$ be a \emph{matched pair}. This means that $G_1$ and $G_2$ are closed subgroups of 
a locally compact group $G$,
such that $G_1\cap G_2=\{e\}$  and such that $G_1G_2$ is of full Haar measure in $G$. 
This datum allows  to define a locally compact quantum group, called the \emph{bicrossed product} \cite{BS1,BSV,VV}. 
We are not going to describe this quantum group, but just its multiplicative unitary. 

Consider the measurable maps  $p_j:G\to G_j$, $j=1,2$,    defined for almost all $x\in G$ by the relation:
$$
p_1(x)p_2(x)=x.
$$
The multiplicative unitary of the  bicrossed product associated to the matched pair $(G_1,G_2; G)$ is the unitary operator $W$ on $L^2(G\times G)$
given by 
\begin{align}
\label{WMP}
W\xi(x,y):=d_w^{1/2}(x,y)\,(\xi\circ w)(x,y),
\end{align}
where $w:G\times G\to G\times G$ is  the pentagonal transformation given by
$$
w(x,y)=\big(xp_1(p_2(x)^{-1}y), p_2(x)^{-1}y\big),
$$
and $d_w$ is the Radon-Nikodym derivative of $w$. Recall that a
 \emph{pentagonal transformation} is a measure class isomorphism $w:X\times X\to X\times X$, where $X$ is
a standard measure space, satisfying the pentagonal relation $ w_{23}\circ w_{13}\circ w_{12}=w_{12}\circ w_{23}$.

A very important result due to Baaj and Skandalis \cite{BS3} (see \cite[Proposition 5.1]{BSV} for the final version of the statement)
is that a pentagonal transformation is  always associated to a matched pair.
 More precisely, given a pentagonal transformation $w$ on $ X$,
let $w_j:X\times X\to X$, $j=1,2$, be such that $w(x,y)=(w_1(x,y),w_2(x,y))$. Assuming that $(w_1, {\rm id})$ and $({\rm id},w_2)$
are measure class isomorphisms too, then there exist a matched pair $(G_1,G_2;G)$, two commuting actions  of the groups $G_1$ (on the right) and  $G_2$
(on the left) on the space $X$, and a measurable $G_1\times G_2$-equivariant  map $f:X\to G$ such that we have almost everywhere:
 $$
 w(x,y)=\big(x.p_1(p_2(f(x))^{-1}.y), p_2(f(x))^{-1}.y\big).
 $$

Associated to a pentagonal transformation $w$
satisfying the extra conditions displayed above or, equivalently, associated to a matched pair $(G_1,G_2;G)$,
there is a cohomology theory, called the measurable Kac cohomology \cite{BSV2}, which in degree 2 is very easy to describe. 
To a  measurable function $\Theta:X\times X\to \T$, we can associate the unitary operator $W_\Theta$ on $L^2(X\times X)$, given by
$$
W_\Theta\xi(x,y):=\Theta(x,y)\,d_w^{1/2}(x,y)\,(\xi\circ w)(x,y).
$$
Such a function $\Theta$ is called a pentagonal $2$-cocycle if $W_\Theta$ still satisfies the pentagonal equation. Moreover, there is  a locally compact 
quantum group, called the \emph{cocycle bicrossed product}, whose multiplicative unitary is $W_\Theta$ \cite{VV}.

The set of  pentagonal $2$-cocycles
forms an Abelian group under pointwise multiplication. Given a measurable function $a:X\to\T$, the function $\Theta_a:=\frac{a\otimes a}{(a\otimes a)\circ w}$
is always a pentagonal $2$-cocycle. These cocycles are considered as trivial since then $W_{\Theta_a}=(a\otimes a) W(a^*\otimes a^*)$.
The \emph{pentagonal $2$-cohomology} of the pentagonal transformation $w$ is defined to be the group of pentagonal $2$-cocycles
divided by the subgroup of trivial $2$-cocycles. We denote this quotient group by $H^2(w,\T)$.

\subsection{The class of groups}
Consider a semidirect product $G=Q\ltimes V$, where $V$ is a (nontrivial) locally compact
Abelian group and $Q$ is a closed subgroup of ${\rm Aut}(V)$. We further assume that the  group
$Q\ltimes V$ satisfies the following properties:
\begin{assumption}
$(1)$ There exists an element $\xi_0\in \hat V$ such that the map
\begin{align}
\label{map:phi}
\phi:Q\to \hat V,\quad q\mapsto q^\flat \xi_0,
\end{align}
is a measure class isomorphism.\\
$(2)$ The group $Q$ has non trivial (measurable) cohomology in degree $2$.
\end{assumption}

In \cite{BGNT3}, condition $(1)$ was called \emph{the dual orbit condition} and
a large  class of examples was given when $V$ is a finite dimensional vector space over a local field.
Recall that a local field $\K$ is a non-discrete locally compact topological field. Local fields
are completely classified: a commutative local field can be either $\R$, $\C$ (Archimedean case), 
or a finite degree extension of $\mathbb Q_p$, or $\mathbb F_q((X))$ (non-Archimedean case)
and a skew local field is a finite dimensional division algebra over a commutative local field. 
For example, one can take:
$$
V=\operatorname{Mat}_n(\mathbb K)\quad \mbox{and} \quad  Q=\operatorname{GL}_n(\mathbb K),
$$
which, for $n=1$, gives  the affine group $\mathbb K^*\ltimes\mathbb K$. 
 \begin{remark} 
Condition $(1)$  implies  the following change of variable formula:
 \begin{align}
 \label{CVF}
  \int_Q f(\phi(q))\frac{dq }{|q|_V}=\int_{\hat V} f(\xi)\,d\xi ,\quad \forall f\in L^1(\hat V).
\end{align}
\end{remark}

We now turn to examples where condition $(2)$ is satisfied. 
\begin{example}
Assume that $Q=(A\times\hat A)\ltimes B$, where $A$ is Abelian and $B$ is arbitrary.
In this case  the inflation homomorphism $H^2(A\times\hat A,\T)\to H^2(Q,\T)$ associated to the projection 
$Q\to A\times\hat A$ is injective. Therefore, one may see the Heisenberg skew-bicharacter $\omega_H$ on $A\times\hat A$ as a 
nontrivial $2$-cocycle on $Q$. 
 Recall that $\omega_H$   is defined  for $q_j=(a_j,\eta_j)\in A\times \hat A$, $j=1,2$, by:
$$
\omega_H(q_1,q_2):=e^{i\langle \eta_1,a_2\rangle}e^{-i\langle \eta_2,a_1\rangle}.
$$
This situation covers the case where $Q=\R^*\times\R^*$, $V=\R\times \R$ and $Q=\C^*\times\C^*$, $V=\C\times \C$.
Indeed, in both cases one has $Q\simeq\R^2\times K$, where $K$ is compact.
\end{example}

\begin{example}
In the previous  example, we cannot  avoid taking two copies of the base field, since $H^2(\R^*,\T)=H^2(\C^*,\T)=\{1\}$.
But this is not the case for a commutative non-Archimedean  local field. Indeed, it is proven in \cite[Lemma 4.1]{Moore2}
that when $\K$ is a finite degree extension of $\mathbb Q_p$, or $\K=\mathbb F_q((X))$, we have
$$
H^2(\K^*,\T)=\hat R\times\hat U_1\times H^2(U_1,\T).
$$
Here $U_1$ is the group of principal units of $\K$ and $R$ is the group of roots of unity in $\K$ of order prime to the characteristic of the residue field.
Therefore, the affine groups  $G=\K^*\ltimes \K$ of the  commutative and non-Archimedean  local fields give nontrivial examples where condition $(2)$ is satisfied.
\end{example}

\subsection{Kohn-Nirenberg quantization for genuine representations}

From Mackey theory, we    know that $\pi:={\rm Ind}_V^G(\xi_0)$ is the unique square-integrable irreducible unitary representation of $G=Q\ltimes V$.
 Realized
on $\CH:= L^2(Q,|q|_V^{-1}dq )$, we have:
\begin{align}
\label{def:pi}
\pi(q,v)\vf(\tilde q)=|q|^{1/2}_V\,e^{-i\langle \tilde q^\flat\xi_0,v\rangle}\,\vf(q^{-1}\tilde q),\quad \forall (q,v)\in G,\;\forall \vf\in\CH.
\end{align}
This representation is unitarily equivalent to the canonical unitary representation on $L^2(V)$ considered in \cite{BGNT3}.  (This formula 
appears in the proof of \cite[Lemma 3.4]{BGNT3} without the factor $|q|^{1/2}_V$. This is because  there the representation ${\rm Ind}_V^G(\xi_0)$
is
realized on $L^2(Q)$, not on $L^2(Q,|q|_V^{-1}dq)$.)

We now pass in review what has been proven in \cite{BGNT3}. First,  under our choice of realization of $\pi$, the (unitary and $G$-equivariant) quantization map
$\Op:L^2(G)\to {\rm HS}(\CH)$ is given by:
\begin{align}
\label{OP}
\Op(f)\zeta(q_0)=\int_Q
(\CF_{\hat V}^*f)(q_0, q^\flat\xi_0- q_0^\flat\xi_0)\,\zeta(q)\,\frac{dq}{|q|_V},\quad \forall f\in L^2(G),\,\forall \zeta\in\CH.
\end{align}
It is called in \cite{BGNT3} the Kohn-Nirenberg quantization of $G$ because under the measure class isomorphism $Q\ltimes V\simeq \hat V\times V$,
$(q,v)\mapsto(\phi(q),v)$,
this is the true Kohn-Nirenberg quantization.

The associated dual unitary  $2$-cocycle $\Omega\in W^*(G\times G)$ has a very simple
(weak) integral representation (see \cite[Equation (3.17) \& Lemma 3.10]{BGNT3}):
$$
\Omega=\int_G e^{-i\langle q^\flat\xi_0-\xi_0,v\rangle}\,\lambda_{(1,v)^{-1}}\otimes\lambda_{(q,0)^{-1}}\, \frac{dqdv}{|q|_V}.
$$
In fact, in \cite{BGNT3},  this expression is not derived from the formula \eqref{eq:Omega} but from a somehow more direct approach.
We will see in Subsection \ref{TSP} that
both constructions give exactly the same dual $2$-cocycle. One important benefit of this formula
is that it 
shows in the case where $G=\R^*\times \R$, that the   locally compact quantum group $(W^*(G),\hat\Delta_\Omega)$
is isomorphic to the one constructed by Stachura \cite{Stachura} using groupoid $C^*$-algebras methods (see \cite[Section 3.6]{BGNT3}).

In order to understand better the  locally compact quantum group $(W^*(G),\hat\Delta_\Omega)$, it is convenient to compute its multiplicative
unitary \eqref{DMU}. For this one needs $\tilde J$, the modular involution of the canonical weight \eqref{TVF}. After a rather tedious analysis of the (isomorphic)
$G$-Galois object $(L^\star(L^2(G))'',\Ad\lambda)$, the following formula was proven in \cite[Proposition 3.24 \& Lemma 3.25]{BGNT3}:
$$
\tilde J J= \int_G e^{i\langle q^\flat\xi_0-\xi_0,v\rangle}\, \Delta_G(q,v)^{-1/2}\,\lambda_{(q,v)}\,\frac{dqdv}{|q|_V},
$$
where the integral has to be understood in the weak sense.

It turns out that, up to  conjugation by the
 partial Fourier transform $\CF_V:L^2(G)\to L^2(Q\times\hat V,|q|_V^{-1}dqd\xi)$, the multiplicative unitary 
$\hat W_{\Omega}$ of the deformed quantum group $(W^*(G),\hat\Delta_\Omega)$ is associated to the following pentagonal transformation  on $Q\times\hat V$
(see \cite[Theorem 3.26]{BGNT3}):
\begin{multline}
\label{PTO}
w(q_1,\xi_1;q_2,\xi_2):=\\
\big(q_2q_1,q_2^\flat\xi_1;\phi^{-1}({q_2^{-1}}^\flat\xi_0+\xi_1)^{-1}\phi^{-1}(\xi_0+\xi_1),{\phi^{-1}({q_2^{-1}}^\flat\xi_0+\xi_1)^{-1}}^\flat
({q_2^{-1}}^\flat\xi_2-\xi_1)\big).
\end{multline}
It is easy to see that this pentagonal transformation satisfies the extra conditions displayed in subsection \ref{CBPPC}. Therefore,
we can apply \cite[Proposition 5.1]{BSV} to get that $(W^*(G),\hat\Delta_\Omega)$  is isomorphic to a bicrossed product quantum group.
It follows from \cite[Theorem 4.1]{BGNT3} that the associated  match pair $(G_1,G_2)$ consists in the two subgroups of 
the dual semidirect product $Q\ltimes\hat V$ given by
$$
G_1=Q,\quad G_2=(1,\xi_0) Q (1,\xi_0)^{-1}.
$$
For $G=\R^*\ltimes \R$, this  shows that the quantum $ax+b$ groups of Stachura and of Baaj-Skandalis are indeed isomorphic. 
Moreover, it entails that the dual-cocycle deformed quantum group $(W^*(G),\hat\Delta_\Omega)$ is self-dual, noncompact, 
nondiscrete, nonunimodular (if the group $G$ is nonunimodular) and with nontrivial scaling group and scaling constant $1$.

\section{Kohn-Nirenberg quantization for projective representations}

\subsection{A class of projective representations}

In order to  incorporate a nontrivial projective representation in the construction of \cite{BGNT3}, one starts with the observation
that the restriction to $Q$ of the representation  \eqref{def:pi} of $G=Q\ltimes V$  coincides with the left regular representation (up to an obvious 
unitary equivalence coming from our choice of measure). It is therefore natural
to modify the representation $\pi$ by inserting a $2$-cocycle $\omega\in Z^2(Q,\T)$ in such a way that the restriction to $Q$
of the new representation $\pi_\omega$   is (equivalent to) the $\omega$-twisted regular representation on $\CH:=L^2(Q,|q|_V^{-1}dq)$:
\begin{align}
\label{piomega}
\pi_\omega(q,v)\zeta( \tilde q):
=|q|^{1/2}_V\,e^{-i\langle \tilde q^\flat\xi_0,v\rangle}\,\omega(q, q^{-1}\tilde q)\,\zeta(q^{-1}\tilde q),\quad \forall(q,v)\in G,\quad 
\forall \tilde q\in Q, \quad \forall\zeta\in\CH.
\end{align}
The cocycle relation 
\begin{align}
\label{CE}
\omega(q_1,q_2)\,\omega(q_1q_2,q_3)=\omega(q_2,q_3)\,\omega(q_1,q_2q_3), \quad \forall q_1,q_2,q_3\in Q,
\end{align}
of the Borel function $\omega: Q\times Q\to \mathbb T$,  implies that $\pi_\omega$ is indeed a projective representation:
$$
\pi_\omega(g)\,\pi_\omega(g')=\omega(q,q')\,\pi_\omega(gg'),
\quad\forall g=(q,v),g'=(q',v')\in G.
$$
Since the inflation homomorphism 
\begin{align}
\label{HQHG}
H^2(Q,\T)\to H^2(G,\T),
\end{align}
associated with the
projection $G\to Q$, $(q,v)\mapsto q$, is injective, we will freely see a nontrivial $2$-cocycle $\omega$ on $Q$ as a nontrivial $2$-cocycle
on $G$.

We shall not  assume that the cocycle $\omega$ is normalized (in the sense that
$\omega(q,q^{-1})=1$ for all $q\in Q$). Nevertheless 
 we will make the absolutely irrelevant normalization:
$$
\omega(e,e)=1.
$$ 
Note that the cocycle relations for $(e,e,q)$ and $(q,e,e)$ imply then:
\begin{align}
\label{NOR}
\omega(q,e)=\omega(e,q)=1,\quad\forall q\in Q.
\end{align}

We first show that the action $\Ad\pi_\omega$ on $B(\CH)$ is ergodic and integrable (see the discussion in Subsection \ref{GO} or \cite[Remark 2.2]{BGNT3}): 
\begin{lemma}
\label{DMomega}
The projective representation $\pi_\omega$   is irreducible and square-integrable.
 The Duflo-Moore operator  is the operator of multiplication by the function
$\Delta_G^{-1}\big|_{Q}$.
\end{lemma}
\bp
Note that $\pi_\omega(1,v)=\pi(1,v)$ is the operator of multiplication by the function $[q\mapsto e^{-i\langle \tilde q^\flat\xi_0,v\rangle}]$.
Therefore, under the identification $Q\simeq \hat V$,  we see by Fourier transform
that  a bounded operator  which commutes with the restriction of $\pi_\omega$
to $V$, must be an operator of multiplication by a function. 
Since $\pi_\omega(g)=\pi(g)\,\omega(q,.)$,
we see that this operator (of multiplication by a function) also commutes 
with the genuine representation $\pi$. Hence, this function must be constant since $\pi$ is irreducible. Therefore, $\pi_\omega$ is irreducible too.

Concerning square-integrability, observe that for $\zeta_1,\zeta_2\in C_c(Q)$, we have:
\begin{align*}
\langle\zeta_1,\pi_\omega(g)\zeta_2\rangle&=|q|^{1/2}_V\int_Q\overline{\zeta_1}(q_0)
\,e^{-i\langle q_0^\flat\xi_0,v\rangle}\,\omega(q, q^{-1}q_0)\,\zeta_2(q^{-1}q_0)\,\frac{dq_0  }{|q_0|_V}\\
& =|q|^{1/2}_V\int_{\hat V}\overline{\zeta_1}(\phi^{-1}(\xi))
\,e^{-i\langle\xi,v\rangle}\,\omega(q, q^{-1}\phi^{-1}(\xi))\,\zeta_2(q^{-1}\phi^{-1}(\xi))\, d\xi  .
\end{align*}
Therefore, we have $\langle\zeta_1,\pi_\omega(g)\zeta_2\rangle=\CF_{\hat V} f_{q,\omega}(v)$, where we have defined:
$$
f_{q,\omega}(\xi):=|q|^{1/2}_V\,\overline{\zeta_1}(\phi^{-1}(\xi))
\,\omega(q, q^{-1}\phi^{-1}(\xi))\,\zeta_2(q^{-1}\phi^{-1}(\xi)).
$$
Since $|f_{q,\omega}|=|f_{q,1}|$ we deduce from the Plancherel formula for $V$ that
$$
\int_G\big|\langle\zeta_1,\pi_\omega(g)\zeta_2\rangle\big|^2dg =
\int_G\big|\langle\zeta_1,\pi(g)\zeta_2\rangle\big|^2dg .
$$
This shows that the Duflo-Moore operators of $\pi_\omega$ and  $\pi$ are equal. The expression for
this operator follows from \cite[Lemma 3.4]{BGNT3}.
\ep

From Lemma \ref{DMomega} and equation  \eqref{vf-DM}, we deduce that  the canonical
 weights $\vf_\omega$ and $\vf$ associated with the actions $\Ad\pi_\omega$ and $\Ad\pi$  on $B(\CH)$ are equal. 
 In particular, this implies that the associated Galois maps
$\G_\omega$ and $\G$  operate on the same GNS space. 
We will take benefits of this fact in our next result, where we prove that $( B(\CH),\Ad\pi_\omega)$ is indeed a $G$-Galois object.
By \cite[Theorem 2.4]{BGNT3}, this  will imply that the twisted group von Neumann algebra $W^*(G,\omega)$ is a type $I$ factor and thus 
the map \eqref{HQHG} will therefore define an injective group homomorphism:
$$
H^2(Q,\T)\to H^2_{I}(G,\mathbb T).
$$

\begin{proposition}
\label{TI}
The  Galois map $\mathcal G_\omega$ associated with the action $\Ad\pi_\omega$   on $B(\CH)$ is unitary.
\end{proposition}
\bp
If $x\in B(\CH)$  is a kernel operator
with kernel $k\in C_c(Q\times Q)$, then $xD^{1/2}$ is also a kernel operator
with kernel $k(1\otimes\Delta_G|_Q^{-1/2})\in C_c(Q\times Q)$. Since $\Delta_G|_Q(q)=|q|_V^{-1}\Delta_Q(q)$, one can therefore take
for the
GNS space of the canonical weight $\vf=\vf_\omega$ the Hilbert space:
$$
L^2\big(Q\times Q,\nu\otimes\tilde \nu\big)\quad\mbox{where}\quad d\nu(q)=|q|_V^{-1}dq
\quad\mbox{and}\quad d\tilde \nu(q)=\Delta_Q^{-1}(q)dq ,
$$
and the GNS map is just the map which associates to a kernel operator  its kernel. Since the GNS map intertwines $\Ad\pi_\omega$
with $\pi_\omega\otimes\pi^c_\omega$, the isometric Galois map \eqref{GM-def}:
\begin{align*}
{\mathcal G}_\omega: L^2(Q\times Q,\nu\otimes\tilde\nu)\otimes L^2(Q\times Q,\nu\otimes\tilde\nu)\to
L^2\big(G,L^2(Q\times Q,\nu\otimes\tilde\nu)\big),
\end{align*}
reads for $k_1,k_2\in C_c(Q\times Q)$:
\begin{align*}
&{\mathcal G}_\omega(k_1\otimes k_2)(g;q_1,q_2)=\int_Q\big((\pi_\omega(g)\otimes\pi_\omega^c(g))k_1(q_1,q_0)\big)\,k_2(q_0,q_2)\,
\frac{dq_0  }{|q_0|_V}\\
&=\int_Q
|q|_V\,e^{-i\langle q_1^\flat\xi_0-q_0^\flat\xi_0,v\rangle}\,\omega(q, q^{-1}q_1)\overline\omega(q, q^{-1}q_0)\,k_1(q^{-1}q_1,q^{-1}q_0)
\,k_2(q_0,q_2)\,
\frac{dq_0  }{|q_0|_V}\\
&=\int_{\hat V}
|q|_V\,e^{i\langle \xi,v\rangle}\,\omega(q, q^{-1}q_1)\overline\omega(q, q^{-1}\phi^{-1}(q_1^\flat\xi_0+\xi))\,
k_1(q^{-1}q_1,q^{-1}\phi^{-1}(q_1^\flat\xi_0+\xi))\\
&\hspace{9.5cm} k_2(\phi^{-1}(q_1^\flat\xi_0+\xi),q_2)\,
d\xi .
\end{align*}
Hence, we get:
$$
\G_\omega= (\CF_V^*\otimes 1)U_\omega (\CF_V\otimes 1)\G,
$$
where $U_\omega$ is the unitary operator of multiplication
by the function 
$$
(q,\xi,q_1,q_2)\mapsto \omega(q, q^{-1}q_1)\overline\omega(q, q^{-1}\phi^{-1}(q_1^\flat\xi_0+\xi)),
$$
and $\G$ is the Galois map associated with the action $\Ad\pi$ on $B(\CH)$.
This concludes the proof since by 
  \cite[Proposition 2.24]{BGNT3}, $W^*(G)$ is a type $I$ factor which by  \cite[Theorem 2.4]{BGNT3} implies that $\G$ is unitary.
\ep

\subsection{The quantization map}
We now explain how to modify the Kohn-Nirenberg quantization \eqref{OP} in order that it becomes covariant
with respect to $\pi_\omega$. At the formal level, it is easy to see that any  quantization map $\Op_\omega:L^2(G)\to \HS(\CH)$
satisfying the covariance relation
$$
\Ad\pi_\omega(g)\circ\Op_\omega=\Op_\omega\circ\lambda_g ,
$$
must be of the following form (for suitable functions $\zeta_1,\zeta_2$ on $Q$ and $f$ on $G$):
$$
\langle \zeta_1,\Op_\omega(f)\zeta_2\rangle = \int f(g)\; T\big(\overline{\pi_\omega(g)^*\zeta_1}\otimes \pi_\omega(g)^*\zeta_2\big)\,dg,
$$
where $T$ is a suitable linear functional. For
the Kohn-Nirenberg quantization $\Op$ given in \eqref{OP}, we find that $T=\delta_e\otimes \nu$ where $\delta_e$ is the Dirac mass 
 at the neutral element  of $Q$
and $\nu$ is the integral on $Q$ against the measure $|q|_V^{-1}dq$. Indeed,  for $f\in C_c(G)$ and $\zeta_1,\zeta_2\in C_c(Q)$, we see by
\eqref{def:pi} that the map
$(g,q)\mapsto f(g)\,\overline{\pi(g)^*\zeta_1}(e)\,\pi(g)^*\zeta_2(q)$ belongs to $C_c(G\times Q)$. Therefore, we can use the Theorem of Fubini to get
$$
\langle \zeta_1,\Op(f)\zeta_2\rangle = \int_{G} f(g)\,\overline{( \pi(g)^*\zeta_1)}(e)  \bigg(\int_Q\pi(g)^*\zeta_2(q_0)  \frac{dq_0  }{|q_0 |_{V}}\bigg)
\, dg.
$$
It is therefore natural to take the following  initial definition 
of the Kohn-Nirenberg quantization of the semidirect product $G=Q\ltimes V$ for the projective representation $\pi_\omega$:

\begin{definition}
For $f\in C_c(G)$, let $\Op_\omega(f)$ be the sesquilinear form on $C_c(Q)$ given by:
 $$
\Op_\omega(f)[\zeta_1 ,\zeta_2 ]:=\int_{G} f(g)\,\overline{( \pi_\omega(g)^*\zeta_1 )}(e)  \bigg(\int_Q\pi_\omega(g)^*\zeta_2 (q_0)  \frac{dq_0  }{|q_0 |_{V}}\bigg)
\, dg.
$$
\end{definition}

\begin{proposition}
\label{UEQ}
For $f\in C_c(G)$ and $\zeta_1 ,\zeta_2 \in C_c(Q)$, we have
$$
\Op_\omega(f)[\zeta_1 ,\zeta_2 ]=\int_{Q\times Q}\overline\zeta_1 (q_0)\,K_\omega(f)(q_0,q)\,\zeta_2 (q)\,
 \frac{dq_0  }{|q_0 |_{V}} \frac{dq }{|q |_{V}},
 $$
 where
 \begin{align}
 \label{Komega}
 K_\omega(f)(q_0,q)= \overline{\omega}(q_0,q_0^{-1}q)\,(\CF_{\hat V}^*f)(q_0, q^\flat\xi_0- q_0^\flat\xi_0).
 \end{align}
 Consequently, the quantization map $\Op_\omega$ extends to a unitary operator from $L^2(G)$
 to ${\HS}(\CH)$ which intertwines the representations $\lambda$ and $\Ad\pi_\omega$ of $G$.
\end{proposition}
\bp
The intertwining property is a direct consequence of the definition of the quantization map as a sesquilinear form.
Since
$$
\pi_\omega(q,v)^*\zeta(q_0)
=|q|_V^{-1/2}\,e^{i\langle (qq_0)^\flat\xi_0,v\rangle}\,\overline{\omega}(q, q_0)\,\zeta(qq_0),
$$
we get for continuous compactly supported $f,\zeta_1,\zeta_2$
(and remembering our normalization \eqref{NOR}):
\begin{align*}
 \Op_\omega(f)[\zeta_1 ,\zeta_2 ]
&=\int f(q,v)\,
e^{-i\langle q^\flat\xi_0,v\rangle}\,\overline\zeta_1 (q)
   \,\bigg(\int e^{i\langle (qq_0)^\flat\xi_0,v\rangle}\,\overline{\omega}(q, q_0)\,\zeta_2 (qq_0)\,
 \frac{dq_0  }{|q _0|_{V}}\bigg)
\, \frac{dq \, dv }{ |q|_V^2}\\
&=\int f(q,v)\,
e^{-i\langle q^\flat\xi_0,v\rangle}\,\overline\zeta_1 (q)
   \,\bigg(\int e^{i\langle q_0^\flat\xi_0,v\rangle}\,\overline{\omega}(q, q^{-1}q_0)\,\zeta_2 (q_0)\,
 \frac{dq_0  }{|q _0|_{V}}\bigg)
\, \frac{dq  \,dv }{ |q|_V}.
\end{align*}
Thus, we deduce by the Theorem of Fubini:
\begin{align*}
 &\Op_\omega(f)[\zeta_1 ,\zeta_2 ]
 =\int \overline\zeta_1 (q)\,
\overline{ \omega}( q,q^{-1}q_0)\,(\CF_{\hat V}^*f)(q, q_0^\flat\xi_0- q^\flat\xi_0)\,\zeta_2 (q_0)\,
 \frac{dq_0  }{|q _0|_{V}}
\, \frac{dq }{ |q|_V},
\end{align*}
which is the formula we need.
\ep

Proposition \ref{UEQ}  (and a small computation) allows  to express the quantization map $\Op_\omega$ in terms of
the quantization map $\Op$ associated with the genuine representation $\pi$:
\begin{lemma}
\label{Vomega}
For $f\in L^2(G)$, we have
$$
\Op_\omega(f)=\Op(\mathcal V_\omega f),
$$
where $\mathcal V_\omega$ is the unitary element of the commutative von Neumann algebra $\CF_V^* L^\infty(Q\times\hat V)\CF_V$ given by
$$
\CF_V\mathcal V_\omega f(q,\xi)= \overline\omega\big(q,\phi^{-1}(\xi_0-{q^{-1}}^\flat\xi)\big)\CF_{ V} f(q,\xi),
\quad\forall f\in L^2(G).
$$
\end{lemma}

The following explains how  the quantization  behaves under the adjoint map:
\begin{lemma}
\label{OU}
For $f\in L^2(G)$, we have
$$
\Op_\omega(f)^*=\Op_\omega(\mathcal U_\omega Jf),
$$
where $\mathcal U_\omega$ is the unitary operator given by:
$$
\CF_V \,\mathcal U_\omega\,f(q,\xi)=\omega\big(\phi^{-1}(\xi_0-{q^{-1}}^\flat\xi),\phi^{-1}(\xi_0-{q^{-1}}^\flat\xi)^{-1}\big)\,
\CF_Vf \big(\phi^{-1}(q^\flat\xi_0-\xi),\xi\big),\quad\forall f\in L^2(G).
$$
 \end{lemma}
\bp
Define $\mathcal U_\omega$ and $\mathcal U$ to be the unitary operators on $L^2(G)$ given by
$$
\mathcal U_\omega f:=\Op_\omega^*(\Op_\omega(Jf)^*)
\quad\mbox{and}\quad
\mathcal U f:=\Op^*(\Op(Jf)^*),
$$
where $J$ is the complex conjugation.
By Lemma \ref{Vomega}, we get
$$
\Op_\omega(f)^*=\Op(\mathcal V_\omega f)^*=\Op(\mathcal U J\mathcal V_\omega f)=\Op_\omega(\mathcal V_\omega^*\mathcal U J\mathcal V_\omega f),
$$
and therefore $\mathcal U_\omega =\mathcal V_\omega^*\mathcal U J\mathcal V_\omega J$. By \cite[Lemma 3.23]{BGNT3}, we have
for $f\in L^2(G)$:
$$
\CF_V \,\mathcal U\,f(q,\xi)=
\CF_Vf \big(\phi^{-1}(q^\flat\xi_0-\xi),\xi\big).
$$
Since moreover
$\CF_V Jf (q,\xi)=\overline{\CF_Vf}(q,-\xi)$,
we get
\begin{multline*}
\CF_V \,\mathcal U_\omega\,f(q,\xi)=\\
\omega\big(q,\phi^{-1}(\xi_0-{q^{-1}}^\flat\xi)\big)
\omega\big(q\phi^{-1}(\xi_0-{q^{-1}}^\flat\xi),\phi^{-1}(\xi_0-{q^{-1}}^\flat\xi)^{-1}\big)\,
\CF_Vf \big(\phi^{-1}(q^\flat\xi_0-\xi),\xi\big),
\end{multline*}
and the formula follows from the cocycle relation for $(q,\phi^{-1}(\xi_0-{q^{-1}}^\flat\xi),\phi^{-1}(\xi_0-{q^{-1}}^\flat\xi)^{-1})$.
\ep

\begin{remark}
It can be shown that $\mathcal U_\omega$ belongs to $W^*(G)'$. Indeed, a simple modification of the computations  given
in \cite[Lemma 3.25]{BGNT3}, shows that  we have (with absolutely convergent integrals) for $f\in C_c(G)$: 
$$
\big(\mathcal U_\omega f\big)(q,v)=\int_G
e^{i\ang*{q'^\flat\xi_0-\xi_0,v'}}\;
\omega(q',q'^{-1})\;
\Delta_G^{-1/2}(q',v')\;
\big(\rho_{(q',v')}f\big)(q,v) \;
\frac{dq'dv'}{|q'|_V}.
$$
Moreover $\mathcal U_\omega=\mathcal U$ when $\omega$ is a normalized cocycle (i.e$.$ when $\omega(q,q^{-1})=1$, $\forall q\in Q$).
\end{remark}

\subsection{The dual $2$-cocycle}
Since   the twisted group von Neumann algebra $W^*(G,\omega)$ is a type $I$ factor
and since we dispose of  an equivariant unitary quantization map  $\Op_\omega$, we know by
\cite[Proposition 2.9]{BGNT3} that a dual unitary $2$-cocycle on $G=Q\ltimes V$ is defined by the formula:
\begin{align}
\label{DBTF}
\Omega_\omega:=(\CJ\otimes\CJ)\tilde\G_\omega^*(1\otimes\CJ)\hat W,
\end{align}
where we recall that $\CJ=J\hat J$ and  $\tilde\G_\omega$ is the Galois map  \eqref{eq:Galois} of the $G$-Galois object
$(B(\CH),\Ad\pi_\omega)$, transported to $L^2(G\times G)$ by the quantization map $\Op_\omega$.
To compute $\Omega_\omega$, our first task is to find an explicit expression for  $\tilde\G_\omega$.

\begin{lemma}
\label{Gomega}
For $f\in L^2(Q\times\hat V\times Q\times\hat V)$, we have:
\begin{multline*}
(\CF_V\otimes \CF_V)\tilde\G_\omega(\CF_V^*\otimes \CF_V^*)f(q_1,\xi_1;q_2,\xi_2)=|q_1|_V\,\Delta_G^{-1/2}(q_1,0)\,
\Delta_G^{1/2}\big(\phi^{-1}(q_2^\flat\xi_0+\xi_1),0\big)
\\
   \overline{\omega}(  \phi^{-1}(\xi_0+{q_2^{-1}}^\flat\xi_1),\phi^{-1}(\xi_0+{q_2^{-1}}^\flat\xi_1)^{-1}
  \phi^{-1}(\xi_0-{q_2^{-1}}^\flat\xi_2)\big)\\
f \big(q_1^{-1}q_2,-{q_1^{-1}}^\flat\xi_1;\phi^{-1}(q_2^\flat\xi_0+\xi_1),\xi_1+\xi_2\big).
\end{multline*}
\end{lemma}
\bp
Let us denote by $\CF C_c(G)$ the dense subset of $L^2(G)$ consisting in functions of the form $\CF_{\hat V}\vf$ with $\vf\in C_c(Q\times \hat V)$.
Note that for $f\in \CF C_c(G)$,  the operator 
$\Op_\omega(f)D^{-1/2}$ extends to a bounded operator. Indeed, by Proposition \ref{UEQ} we deduce that $K_\omega(f)\in
C_c(Q\times Q)$ and
by Lemma \ref{DMomega} that  $\Op_\omega(f)D^{-1/2}$ is a kernel operator with kernel  given by $K_\omega(f)(1\otimes
\Delta_G^{1/2}|_Q)$ which belongs to $C_c(Q\times Q)$ too. 
Therefore, $\Op_\omega(f)D^{-1/2}$ is  Hilbert-Schmidt and thus bounded.

For $f_1,f_2\in \CF C_c(G)$, we  express the Galois map $\tilde \G_\omega$ in terms of the symbol
map $\Op_\omega^*:{\HS}(\CH)\to L^2(G)$ to get:
$$
\tilde \G_\omega(f_1\otimes f_2)(g_1,g_2)=\Op_\omega^*\big(\Ad\pi_\omega(g_1)\big(\Op_\omega(f_1)D^{-1/2}\big)
\Op_\omega(f_2)\big)(g_2).
$$
This formulas  makes perfect good sense since, as  already  observed, $\Op_\omega(f_1)D^{-1/2}$ is bounded and
$\Op_\omega(f_2)$ is Hilbert-Schmidt.
Equivariance of the quantization map
and the fact that the Duflo-Moore operator is $\pi_\omega$-quasi-invariant of weight $\Delta_G$, entail that
$$
\tilde\G_\omega(f_1\otimes f_2)(g_1,g_2)=\Delta_G^{-1/2}(g_1)\Op_\omega^*\big(\Op_\omega(\lambda_{g_1}(f_1))D^{-1/2}
\Op_\omega(f_2)\big)(g_2).
$$
The product formula for operator kernels  shows that
  $\Op_\omega(\lambda_{g_1}(f_1))D^{-1/2}\Op_\omega(f_2)$ has kernel $ k_{g_1}\in L^2(Q\times Q,|q_2|_V^{-1}dq_2|q_3|_V^{-1}dq_3)$, 
  given by
$$
k_{g_1}(q_2,q_3)
= \int K_{\omega}(\lambda_{g_1}f_1)(q_2,q)\,\Delta_G^{1/2}(q,0) \,
 K_\omega(f_2)(q,q_3)\frac{dq }{|q|_V}.
 $$
 Using Proposition \ref{UEQ}, we get:
 \begin{multline*}
 k_{g_1}(q_2,q_3)= \\\int \overline{\omega}(q_2,q_2^{-1}q)\,\overline{\omega}(q,q^{-1}q_3)\, (\CF_{ V}\lambda_{g_1}f_1)(q_2,q_2^\flat\xi_0-q^\flat\xi_0)\,
(\CF_{V}f_2)(q,q^\flat\xi_0-q_3^\flat\xi_0)\,\Delta_G^{1/2}(q,0) \frac{dq }{|q|_V}.
\end{multline*}
Inverting the kernel formula  of Proposition \ref{UEQ}, we deduce:
\begin{align*}
&(1\otimes \CF_V)\tilde \G_\omega(f_1\otimes f_2)(g_1;q_2,\xi_2)=\Delta_G^{-1/2}(q_1,0)\,
\omega\big(q_2,\phi^{-1}(\xi_0-{q_2^{-1}}^\flat\xi_2)\big)\,
 k_{g_1}\big(q_2,\phi^{-1}(q_2^\flat\xi_0-\xi_2)\big)\\
 &\hspace{3cm}=\Delta_G^{-1/2}(q_1,0)\,
\omega\big(q_2,\phi^{-1}(\xi_0-{q_2^{-1}}^\flat\xi_2)\big)
 \int \overline\omega(q_2,q_2^{-1}q)\overline\omega\big(q,q^{-1}\phi^{-1}(q_2^\flat\xi_0-\xi_2)\big) \\
 &\hspace{4cm}(\CF_{ V}\lambda_{g_1}f_1)(q_2,q_2^\flat\xi_0-q^\flat\xi_0)\,
(\CF_{V}f_2)(q,q^\flat\xi_0-q_2^\flat\xi_0+\xi_2)\,\Delta_G^{1/2}(q,0) \frac{dq }{|q|_V}.
\end{align*}
Using the cocycle relation for $(q_2,q_2^{-1}q, q^{-1}q_2\phi^{-1}(\xi_0-{q_2^{-1}}^\flat\xi_2))$, we obtain:
\begin{align*}
&(1\otimes \CF_V)\tilde \G_\omega(f_1\otimes f_2)(g_1;q_2,\xi_2)=\Delta_G^{-1/2}(q_1,0)\,
 \int \overline\omega\big(q_2^{-1}q,q^{-1}\phi^{-1}(q_2^\flat \xi_0-\xi_2)\big) \\
 &\hspace{4cm}(\CF_{ V}\lambda_{g_1}f_1)(q_2,q_2^\flat\xi_0-q^\flat\xi_0)\,
(\CF_{V}f_2)(q,q^\flat\xi_0-q_2^\flat\xi_0+\xi_2)\,\Delta_G^{1/2}(q,0) \frac{dq }{|q|_V}.
\end{align*}
Now, the  formula  (valid  for any $f\in L^2(G)$):
\begin{align}
\label{CFL}
(\CF_V\lambda_{g^{-1}} f)(q',\xi')=|q|_V^{-1} e^{i\langle q^\flat\xi',v\rangle}(\CF_Vf)(qq',q^\flat\xi'),
\end{align}
gives:
\begin{align*}
&(1\otimes \CF_V)\tilde \G_\omega(f_1\otimes f_2)(q_1,v_1;q_2,\xi_2)=\frac{|q_1|_V^{3/2}}{\Delta_Q^{1/2}(q_1)}\,
 \int  \overline\omega\big(q_2^{-1}q,q^{-1}\phi^{-1}(q_2^\flat \xi_0-\xi_2)\big) e^{-i\langle q_2^\flat\xi_0-q^\flat\xi_0,v_1\rangle}\\
 &\hspace{1.7cm}(\CF_{ V}f_1)(q_1^{-1}q_2,(q_1^{-1}q_2)^\flat\xi_0-(q_1^{-1}q)^\flat\xi_0)\,
(\CF_{V}f_2)(q,q^\flat\xi_0-q_2^\flat\xi_0+\xi_2)\,\Delta_G^{1/2}(q,0) \frac{dq }{|q|_V}.
\end{align*}
The $L^2$-Fourier inversion formula   gives then:
\begin{multline*}
(\CF_V\otimes \CF_V)\tilde\G_\omega(f_1\otimes f_2)(q_1,\xi_1;q_2,\xi_2)=\frac{|q_1|_V^{3/2}}{\Delta_Q^{1/2}(q_1)}\,
\Delta_G^{1/2}\big(\phi^{-1}(q_2^\flat\xi_0+\xi_1),0\big)\\
\overline \omega\big(\phi^{-1}(\xi_0+{q_2^{-1}}^\flat\xi_1),\phi^{-1}(\xi_0+{q_2^{-1}}^\flat\xi_1)^{-1}\phi^{-1}(\xi_0-{q_2^{-1}}^\flat\xi_2)\big)
\\
\CF_{ V}f_1(q_1^{-1}q_2,-{q_1^{-1}}^\flat\xi_1)\,
\CF_{V}f_2\big(\phi^{-1}(q_2^\flat\xi_0+\xi_1),\xi_1+\xi_2\big),
\end{multline*}
and the proof follows  by density of $\CF C_c(G)\otimes \CF C_c(G)$ in $L^2(G\times G)$.
\ep

\begin{theorem}
\label{OmegaF}
For $f\in L^2(Q\times\hat V\times Q\times\hat V)$, we have:
\begin{multline*}
\big((\CF_V\otimes\CF_V)\Omega_\omega (\CF_V^*\otimes\CF_V^*)f\big)(q_1,\xi_1;q_2,\xi_2)=
 |\phi^{-1}(\xi_0+\xi_1)|_V^{-1}\\
\omega\big(\phi^{-1}(\xi_0+\xi_1),\phi^{-1}(\xi_0+\xi_2)\big)\,
f\big(q_1, \xi_1 ;\phi^{-1}(\xi_0+\xi_1)q_2,\phi^{-1}(\xi_0+\xi_1)^\flat\xi_2\big).
\end{multline*}
Moreover, the map $\omega\mapsto \Omega_\omega$ induces an embedding of $H^2(Q,\mathbb T)$ into $H_I^2(\hat G,\mathbb T)$.
\end{theorem}
Recall that $H_I^2(\hat G,\mathbb T)$ denotes the set of equivalence classes of dual unitary $2$-cocycles $\Omega$ such that
the twisted dual group von Neumann algebra $W^*(\hat G,\Omega)$ is a type I factor.  

\begin{proof}[Proof of Theorem \ref{OmegaF}]
Starting from  the relations (see \cite[proof of Theorem 3.26]{BGNT3}):
\begin{align}
\label{JJF}
(\CF_V \CJ\CF_V^*\,f)(q,\xi)=|q|_V\,\Delta_G^{-1/2}(q,0)\,f\big(q^{-1},-{q^{-1}}^\flat\xi\big),
\end{align}
and
\begin{equation}\label{WF}
\big((\CF_V\otimes \CF_V)\hat W^* (\CF_V^*\otimes \CF_V^*)f\big)(q_1,\xi_1;q_2,\xi_2)
= |q_2|_V\, f\big(q_2^{-1}q_1,{q_2^{-1}}^\flat\xi_1;q_2,\xi_1+\xi_2\big),
\end{equation}
we get from Lemma \ref{Gomega}
\begin{align*}
&\big((\CF_V\otimes\CF_V)\Omega_\omega^* (\CF_V^*\otimes\CF_V^*)f\big)(q_1,\xi_1;q_2,\xi_2)\\
&=(\CF_V\otimes \CF_V)\hat W^* (1\otimes\CJ)\tilde\G_\omega(\CJ\otimes\CJ)(\CF_V^*\otimes \CF_V^*)f(q_1,\xi_1;q_2,\xi_2)\\
&=|q_2|_V\,(\CF_V\otimes \CF_V) (1\otimes\CJ)\tilde\G_\omega(\CJ\otimes\CJ)(\CF_V^*\otimes \CF_V^*)
f\big(q_2^{-1}q_1,{q_2^{-1}}^\flat\xi_1;q_2,\xi_1+\xi_2\big)\\
&=|q_2|_V^2\,\Delta_G^{-1/2}(q_2,0)\,(\CF_V\otimes \CF_V) \tilde\G_\omega(\CJ\otimes\CJ)(\CF_V^*\otimes \CF_V^*)
f\big(q_2^{-1}q_1,{q_2^{-1}}^\flat\xi_1;q_2^{-1},-{q_2^{-1}}^\flat\xi_1-{q_2^{-1}}^\flat\xi_2\big)\\
&=|q_1q_2|_V\,\Delta_G^{-1/2}(q_1q_2,0)\, \Delta_G^{1/2}\big(\phi^{-1}(\xi_0+\xi_1),0\big)\,
 \overline\omega\big(\phi^{-1}(\xi_0+\xi_1),\phi^{-1}(\xi_0+\xi_1)^{-1}\phi^{-1}(\xi_0+\xi_1+\xi_2)\big) \\
&\hspace{3.4cm}(\CF_V\otimes \CF_V) (\CJ\otimes\CJ)(\CF_V^*\otimes \CF_V^*)
f\big(q_1^{-1},-{q_1^{-1}}^\flat\xi_1;q_2^{-1}\phi^{-1}(\xi_0+\xi_1),-{q_2^{-1}}^\flat\xi_2\big)\\
&=|\phi^{-1}(\xi_0+\xi_1)|_V\,
  \overline\omega\big(\phi^{-1}(\xi_0+\xi_1),\phi^{-1}(\xi_0+\xi_1)^{-1}\phi^{-1}(\xi_0+\xi_1+\xi_2)\big)\\
&\hspace{8cm}
f\big(q_1,\xi_1;\phi^{-1}(\xi_0+\xi_1)^{-1}q_2,{\phi^{-1}(\xi_0+\xi_1)^{-1}}^\flat\xi_2\big).
\end{align*}
Passing to the adjoint (taking care of the fact that we work with the measure $|q|_V^{-1}dq d\xi $),
we get the announced formula.

For the remaining part of the statement, consider $\eta: \hat V\to Q$ and   $\mu:\hat V\times\hat V\to\hat V\times\hat V$,  defined by
$\eta(\xi)=\phi^{-1}(\xi_0+\xi)$ and by
$\mu(\xi_1, \xi_2):=\big( \xi_1 ,\eta(\xi_1)^\flat\xi_2\big)$.
Denote by $\Delta $  the coproduct of $L^\infty(Q)$, $L^\infty(\hat V)$...
(not to be confused with the modular functions) and by $\hat\Delta$, the one of $W^*(V)$, $W^*(G)$...
An easy computation shows that for $u\in L^\infty(Q)$, we have:
$$
(\Delta u)\circ(\eta\times \eta )\circ \mu^{-1}=\Delta(u\circ\eta).
$$
Hence,  if $\omega'=\frac{u\otimes u}{\Delta u}\,\omega$, with $u\in L^\infty(Q)$ unitary,  then we get:
$$
(\CF_V\otimes\CF_V)\Omega_{\omega'} (\CF_V^*\otimes\CF_V^*)=( u\circ\eta\otimes u\circ\eta)(\CF_V\otimes\CF_V)\Omega_\omega (\CF_V^*\otimes\CF_V^*)
\Delta( u\circ\eta)^*.
$$
Since, for $f\in L^\infty(\hat V)$, we have $(\CF_V^*\otimes\CF_V^*)\Delta(f)(\CF_V\otimes\CF_V)=\hat\Delta(\lambda(\CF_V^* f))$,
we deduce
$$
\Omega_{\omega'} =(\lambda(\tilde u)\otimes \lambda(\tilde u))\,\Omega_\omega\,\hat\Delta(\lambda(\tilde u))^*
\quad
\mbox{where}\quad
\tilde u(q,v):=\delta_e(q)\,\CF_V^* (u\circ \eta)(v).
$$
Therefore, we have a well defined map $\tau :[\omega]\in H^2(Q,\T)\mapsto [\Omega_\omega]\in H^2_I(\hat G,\T)$.
To see that this map is injective, consider the diagram
\begin{align*}
\xymatrix{
H^2_I(\hat G,\T)\ar[r]^\mu&H^2_I( G,\T)\\
H^2( Q,\T)\ar[u]_\tau\ar[ur]^{\iota}&
},
\end{align*}
where $\iota$ is the  inflation homomorphism $H^2( Q,\T)\to H^2( G,\T)$ (which takes values in $H^2_I( G,\T)$ by Proposition \ref{TI}) and
$\mu$ is the embedding \eqref{Emb1}. By \cite[Proposition 2.9]{BGNT3}, the $G$-Galois object defined from $\Omega_\omega$
is isomorphic to $(B(L^2(Q)),\Ad \pi_\omega)$.  This shows that the diagram above is commutative, hence the  injectivity of $\tau$.
\end{proof}

Theorem \ref{OmegaF} shows that the dual unitary  $2$-cocycle $\Omega_\omega$ factorizes as:
$$
\Omega_\omega=U_\omega
\,\Omega,
$$
where $\Omega$ is the dual $2$-cocycle attached with the genuine representation $\pi$ (that is, the dual $2$-cocycle constructed in \cite{BGNT3})
and where $U_\omega$ is the  unitary convolution operator given by:
$$
U_\omega:= (\CF_V^*\otimes\CF_V^*)\,\omega\big(\phi^{-1}(\xi_0+. ),\phi^{-1}(\xi_0+. )\big)\,(\CF_V\otimes\CF_V).
$$
It is therefore natural to ask if there exist other dual unitary  $2$-cocycles of the form
\begin{align}
\label{Theta}
\Omega_\Theta:=(\CF_V^*\otimes\CF_V^*)\,\Theta\,(\CF_V\otimes\CF_V)\,\Omega,
\end{align}
where  $\Theta:\hat V\times \hat V\to\mathbb T$ is a measurable function.
We answer this question negatively showing therefore  that our construction exhausts  all possibilities.

\begin{proposition}
Let $\Theta:\hat V\times \hat V\to\mathbb T$ be a measurable function such that the unitary convolution operator $\Omega_\Theta$ satisfies the dual $2$-cocycle
equation \eqref{eq:dual-cocycle}. Then, there exists a group $2$-cocycle $\theta\in Z^2(Q,\mathbb T)$ such that we have almost everywhere:
$$
\Theta(\xi_1,\xi_2)=\theta\big(\phi^{-1}(\xi_0+\xi_1),\phi^{-1}(\xi_0+\xi_2)\big).
$$
\end{proposition}
\bp
Observe that for $u\in L^\infty(\hat V)$,  we  have
$$
\hat\Delta\big(\CF_V^*\,u\,\CF_V)=(\CF_V^*\otimes\CF_V^*)\,\Delta u\,(\CF_V\otimes\CF_V).
$$
From this we get that
\begin{align*}
&(\CF_V\otimes\CF_V\otimes\CF_V)(\Omega_\Theta\otimes1)(\Dhat\otimes\iota)(\Omega_\Theta)(\CF_V^*\otimes\CF_V^*\otimes\CF_V^*)=\\
&\qquad\qquad\qquad\qquad\qquad\qquad \Theta'
\,(\CF_V\otimes\CF_V\otimes\CF_V)(\Omega\otimes1)(\Dhat\otimes\iota)(\Omega)(\CF_V^*\otimes\CF_V^*\otimes\CF_V^*),
\end{align*}
where $\Theta':\hat V\times \hat V\times \hat V\to\mathbb T$ is the measurable function given by:
$$
\Theta'(\xi_1,\xi_2,\xi_3):=\Theta(\xi_1,\xi_2)\,\Theta\big(\xi_1+\phi^{-1}(\xi_0+\xi_1)^\flat\xi_2,\xi_3\big).
$$
Similarly, we have
\begin{align*}
&(\CF_V\otimes\CF_V\otimes\CF_V)(1\otimes\Omega_\Theta)(\iota\otimes\Dhat)(\Omega_\Theta)(\CF_V^*\otimes\CF_V^*\otimes\CF_V^*)=\\
&\qquad\qquad\qquad\qquad\qquad\qquad \Theta''
\,(\CF_V\otimes\CF_V\otimes\CF_V)(1\otimes\Omega)(\iota\otimes\Dhat)(\Omega)(\CF_V^*\otimes\CF_V^*\otimes\CF_V^*),
\end{align*}
where $\Theta'':\hat V\times \hat V\times \hat V\to\mathbb T$ is the measurable function given by:
$$
\Theta''(\xi_1,\xi_2,\xi_3):=\Theta(\xi_2,\xi_3)\,\Theta\big(\xi_1,\xi_2+\phi^{-1}(\xi_0+\xi_2)^\flat\xi_3\big).
$$
Therefore, we deduce that $\Omega_\Theta$ is a dual $2$-cocycle if and only if $\Theta'=\Theta''$.

Define  the measurable function $\theta: Q\times Q\to\mathbb T$ as follows:
$$
\theta(q_1,q_2):=\Theta(q_1^\flat\xi_0-\xi_0,q_2^\flat\xi_0-\xi_0).
$$
In terms of $\theta$, the equation $\Theta'=\Theta''$ is equivalent to:
\begin{align*}
&\theta\big(\phi^{-1}(\xi_0+\xi_1),\phi^{-1}(\xi_0+\xi_2)\big)\theta\big(\phi^{-1}(\xi_0+\xi_1+\phi^{-1}(\xi_0+\xi_1)^\flat\xi_2),\phi^{-1}(\xi_0+\xi_3)\big)=\\
&\qquad\qquad\qquad\theta\big(\phi^{-1}(\xi_0+\xi_2),\phi^{-1}(\xi_0+\xi_3)\big)
\theta\big(\phi^{-1}(\xi_0+\xi_1),\phi^{-1}(\xi_0+\xi_2+\phi^{-1}(\xi_0+\xi_2)^\flat\xi_3)\big).
\end{align*}
Since
$$
\phi^{-1}(\xi_0+\xi_1+\phi^{-1}(\xi_0+\xi_1)^\flat\xi_2)=
\phi^{-1}(\xi_0+\xi_1)\phi^{-1}(\xi_0+\xi_2),
$$
we deduce from our dual orbit condition that $\theta$ satisfies  the  $2$-cocycle equation \eqref{CE} a.e.
\ep

\subsection{The star-product}
\label{TSP}
This section is a preparation to Proposition \ref{MW} where
we determine the modular involution of the canonical weight $\vf$ of the $G$-Galois object  $(W^*(\hat G, \Omega_\omega),\Ad\rho)$.
Our main task is to show that the products \eqref{starOmega} and \eqref{star-prod} coincide on a sufficiently large domain.

So, consider the associative product $\star_\omega$ defined on $L^2(G)$ by:
$$
f_1\star_\omega f_2:=\Op_\omega^*\big(\!\Op_\omega(f_1)\Op_\omega(f_2)\big).
$$
We first give an integral formula for this product. 
To this aim, we let  $\CF \CL(G)$ be  the dense subspace of $L^2(G)$,
consisting in functions $f$  such that $\CF_Vf$ is essentially bounded and essentially zero outside of a
set of finite measure.
\begin{lemma}
\label{LSPF1}
The space $\CF \CL(G)$ is an algebra for the product $\star_\omega$ and
 for any $f_1,f_2\in\CF\CL(G)$ we have almost everywhere:
 \begin{multline}
  \label{SPF1}
\CF_V (f_1\star_\omega f_2)(q,\xi)=
\int_{\hat V}\overline\omega\big(\phi^{-1}(\xi_0-{q^{-1}}^\flat\xi'),\phi^{-1}(\xi_0-{q^{-1}}^\flat\xi')^{-1}\phi^{-1}(\xi_0-{q^{-1}}^\flat\xi)\big)\\
(\CF_V f_1)(q,\xi')
(\CF_V f_2)\big(\phi^{-1}(q^\flat\xi_0-\xi'),\xi-\xi'\big)\,d\xi '.
\end{multline}
\end{lemma}
\bp
Note first that if  $\CF_Vf_1$ and $\CF_Vf_2$ are essentially bounded and essentially zero outside  a set of finite measure, then
the same is true for the RHS of \eqref{SPF1}. Therefore, if the equality \eqref{SPF1} holds, then we will deduce that $(\CF \CL(G),\star_\omega)$
is an algebra.

For $f_1,f_2\in L^2(G)$ we have by Lemma \ref{Vomega}:
$$
f_1\star_\omega f_2=\mathcal V_\omega ^*\big((\mathcal V_\omega f_1)\star (\mathcal V_\omega f_2)\big),
$$
so that:
$$
\CF_V (f_1\star_\omega f_2)(q,\xi)=\omega\big(q,\phi^{-1}(\xi_0-{q^{-1}}^\flat\xi)\big) \,\CF_V\big((\mathcal V_\omega f_1)\star (\mathcal V_\omega f_2)\big)(q,\xi).
$$
For $f_1,f_2\in \CF \CL(G)$, by \cite[Lemma 3.9]{BGNT3}   we have almost everywhere:
\begin{equation*}
\CF_V(f_1\star f_2)(q,\xi)=\int_{\hat V}
(\CF_V f_1)(q,\xi')\,(\CF_V f_2)\big(\phi^{-1}(q^\flat\xi_0-\xi'),\xi-\xi'\big)\,d\xi '.
\end{equation*}
Clearly, the unitary $\mathcal V_\omega$ is an automorphism of $\CF \CL(G)$. We therefore get

\begin{multline*}
\CF_V (f_1\star_\omega f_2)(q,\xi)=\omega\big(q,\phi^{-1}(\xi_0-{q^{-1}}^\flat\xi)\big)\int_{\hat V}
\overline\omega\big(q,\phi^{-1}(\xi_0-{q^{-1}}^\flat\xi')\big)(\CF_V f_1)(q,\xi')\\
\overline\omega\big(q\phi^{-1}(\xi_0-{q^{-1}}^\flat\xi'),\phi^{-1}(\xi_0-{q^{-1}}^\flat\xi')^{-1}\phi^{-1}(\xi_0-{q^{-1}}^\flat\xi)\big)(\CF_V f_2)\big(\phi^{-1}(
q^\flat\xi_0-\xi'),\xi-\xi'\big)\,d\xi '.
\end{multline*}
The  cocycle identity applied to $(q,\phi^{-1}(\xi_0-{q^{-1}}^\flat\xi'),\phi^{-1}(\xi_0-{q^{-1}}^\flat\xi')^{-1}\phi^{-1}(\xi_0-{q^{-1}}^\flat\xi))$ gives the desired formula.
\ep

Recall that the Fourier algebra $A(G)$ consists of functions of the form $\zeta\ast \check \zeta'$ with $\zeta,\zeta'\in L^2(G)$ and,  as a Banach space,
$A(G)$ identifies with the predual $W^*(G)_*$ via the map $\zeta\ast \check \zeta'\mapsto \langle \bar\zeta,\cdot\,\zeta'\rangle$.

\begin{proposition}
\label{ESP}
The space $A(G)\cap L^2(G)$ is an algebra for the product $\star_\omega$. Moreover,  we have
 for all $f_1,f_2\in A(G)\cap L^2(G)$ and all $g\in G$:
\begin{align}
\label{SES}
f_1\star_\omega f_2(g)=(f_1\otimes f_2)\big(\hat\Delta(\lambda_g)\Omega_\omega^*\big).
\end{align}
\end{proposition}
\bp
Let $f_j\in A(G)$, $j=1,2$. This means that there exist $\zeta_j,\zeta_j'\in L^2(G)$, $j=1,2$, such that $f_j=\zeta_j\ast \check \zeta_j'$. Under the identification of
$A(G)$ with $W^*(G)_*$, we then have:
$$
(f_1\otimes f_2)\big(\hat\Delta(\lambda_g)\Omega_\omega^*\big)=\big\langle \lambda_{g^{-1}} \overline{\zeta_1} \otimes\lambda_{g^{-1}} \overline{\zeta_2} ,
\Omega_\omega^*(\zeta_1 '\otimes\zeta_2' )\big\rangle.
$$
In the proof of Theorem \ref{OmegaF} we have already given the formula for $(\CF_V\otimes\CF_V)\Omega_\omega^*(\CF_V^*\otimes\CF_V^*)$.
This formula entails that:
\begin{multline}
\label{EQ1}
(f_1\otimes f_2)\big(\hat\Delta(\lambda_g)\Omega_\omega^*\big)=\int |\phi^{-1}(\xi_0+\xi_1)|_V\,
\overline\omega\big(\phi^{-1}(\xi_0+\xi_1),\phi^{-1}(\xi_0+\xi_1)^{-1}\phi^{-1}(\xi_0+\xi_1+\xi_2)\big)\\
\overline{(\CF_V\lambda_{g^{-1}} \overline{\zeta_1} )}(q_1,\xi_1)\,\overline{(\CF_V\lambda_{g^{-1}} \overline{\zeta_2} )}(q_2,\xi_2)\,
(\CF_V\zeta_1 ')(q_1, \xi_1 )\\
(\CF_V\zeta_2 ')\big(\phi^{-1}(\xi_0+\xi_1)^{-1}q_2,{\phi^{-1}(\xi_0+\xi_1)^{-1}}^\flat\xi_2\big) \, \frac{dq_1 }{|q_1|_V}d\xi _1
 \frac{dq_2}{|q_2|_V}d\xi _2.
\end{multline}
Observe that the integrand in \eqref{EQ1} is absolutely convergent, uniformly in $g\in G$.
Indeed, the $L^1$-norm of the integrand is bounded by $\|\zeta_1 \|_2\|\zeta_1 '\|_2\|\zeta_2 \|_2\|\zeta_2 '\|_2\leq \|f_1\|_{A(G)} \|f_2\|_{A(G)}$.

The  formula  \eqref{CFL} implies that
\begin{multline*}
(f_1\otimes f_2)\big(\hat\Delta(\lambda_g)\Omega_\omega^*\big)=|q|_V^{-2}\int |\phi^{-1}(\xi_0+\xi_1)|_V\,
\overline\omega\big(\phi^{-1}(\xi_0+\xi_1),\phi^{-1}(\xi_0+\xi_1)^{-1}\phi^{-1}(\xi_0+\xi_1+\xi_2)\big)\\
(\CF_V \zeta_1 )(qq_1,-q^\flat\xi_1)\,(\CF_V \zeta_2 )(qq_2,-q^\flat\xi_2)\,
(\CF_V\zeta_1 ')(q_1, \xi_1 )\\
(\CF_V\zeta_2 ')\big(\phi^{-1}(\xi_0+\xi_1)^{-1}q_2,{\phi^{-1}(\xi_0+\xi_1)^{-1}}^\flat\xi_2\big) \,
e^{-i\langle q^\flat(\xi_1+ \xi_2),v\rangle}\, \frac{dq_1 }{|q_1|_V}d\xi _1
 \frac{dq_2}{|q_2|_V}d\xi _2.
\end{multline*}
Performing the affine change of variable $\xi_2\mapsto -{q^{-1}}^\flat\xi_2-\xi_1$ and using the Theorem of Fubini, we get:
\begin{align}
\label{EQ2}
(f_1\otimes f_2)\big(\hat\Delta(\lambda_g)\Omega_\omega^*\big)=(\CF_V^* \tilde f)(q,v),
\end{align}
where  we have defined
\begin{multline*}
\tilde f(q,\xi):=|q|_V\!\!\int |\phi^{-1}(\xi_0+\xi_1)|_V\,
\overline\omega\big(\phi^{-1}(\xi_0+\xi_1),\phi^{-1}(\xi_0+\xi_1)^{-1}\phi^{-1}(\xi_0-{q^{-1}}^\flat\xi)\big)\\
(\CF_V \zeta_1 )(qq_1,-q^\flat\xi_1)
(\CF_V \zeta_2 )(qq_2,\xi+q^\flat\xi_1)\,
(\CF_V\zeta_1 ')(q_1, \xi_1 )\\
(\CF_V\zeta_2 ')\big(\phi^{-1}(\xi_0+\xi_1)^{-1}q_2,-{\phi^{-1}(\xi_0+\xi_1)^{-1}}^\flat({q^{-1}}^\flat\xi+\xi_1)\big) 
\,\frac{dq_1 }{|q_1|_V}d\xi _1
 \frac{dq_2}{|q_2|_V}.
\end{multline*}
As a  consequence of the fact that the integrand of the RHS of the equality \eqref{EQ1} is  integrable uniformly in $g$,
we  deduce that $\tilde f$ belongs  to $L^\infty(Q,L^1(\hat V))$ so that the RHS of \eqref{EQ2} defines a bounded function.
Setting $\xi'=-q^\flat\xi_1$, we get:
\begin{multline*}
\tilde f(q,\xi)=\int |\phi^{-1}(\xi_0-{q^{-1}}^\flat\xi')|_V\,
\overline\omega\big(\phi^{-1}(\xi_0-{q^{-1}}^\flat\xi'),\phi^{-1}(\xi_0-{q^{-1}}^\flat\xi')^{-1}\phi^{-1}(\xi_0-{q^{-1}}^\flat\xi)\big)\\
(\CF_V \zeta_1 )(qq_1,\xi')\,(\CF_V \zeta_2 )(qq_2,\xi-\xi')\,
(\CF_V\zeta_1 ')(q_1, -{q^{-1}}^\flat\xi' )\\(\CF_V\zeta_2 ')\big(\phi^{-1}(\xi_0-{q^{-1}}^\flat\xi')^{-1}q_2,-{\phi^{-1}(q^\flat\xi_0-\xi')^{-1}}^\flat
(\xi-\xi')\big)
\frac{dq_1 }{|q_1|_V}d\xi'
 \frac{dq_2}{|q_2|_V}.
\end{multline*}
Performing the left translations $q_j\mapsto q^{-1}q_j$, $j=1,2$, and using again the Theorem of Fubini, we obtain:
\begin{multline*}
\tilde f(q,\xi)=\int \overline\omega\big(\phi^{-1}(\xi_0-{q^{-1}}^\flat\xi'),\phi^{-1}(\xi_0-{q^{-1}}^\flat\xi')^{-1}\phi^{-1}(\xi_0-{q^{-1}}^\flat\xi)\big)
 |q|_V\\
 \bigg(\int (\CF_V \zeta_1 )(q_1,\xi')\,(\CF_V\zeta_1 ')(q^{-1}q_1, -{q^{-1}}^\flat\xi' )\frac{dq_1 }{|q_1|_V}\bigg)
 |\phi^{-1}(q^\flat\xi_0-\xi')|_V\\\bigg(\int
(\CF_V \zeta_2 )(q_2,\xi-\xi')
(\CF_V\zeta_2 ')\big(\phi^{-1}(q^\flat\xi_0-\xi')^{-1}q_2,-{\phi^{-1}(q^\flat\xi_0-\xi')^{-1}}^\flat(\xi-\xi')\big)
 \frac{dq_2}{|q_2|_V}\bigg)d\xi'.
\end{multline*}
Observe  then that for $f=\vf\ast\check\vf'\in A(G)\cap L^2(G)$, we have
$$
\CF_V(f)(q,\xi)=|q|_V\int (\CF_V\vf)(q_0,\xi) \,(\CF_V\vf')(q^{-1}q_0,-{q^{-1}}^\flat\xi)\,\frac{dq_0  }{|q_0|_V}.
$$
Indeed, note first that $C_c(G)\ast C_c(G)$ is dense in $A(G)\cap L^2(G)$, simultaneously in the norms
of $A(G)$ and $L^2(G)$. This is proven in the second part of the proof of \cite[Lemma 3.11]{BGNT3}. Therefore, it is enough
to prove this equality when $\vf,\vf'\in C_c(G)$ and in this case, the proof  easily  follows by the Theorem of Fubini:
\begin{multline*}
\CF_V(\vf\ast\check\vf')(q,\xi)
=\int e^{-i\langle \xi,v\rangle}\,\vf(q_0,v_0) \,\vf'\big(q^{-1}q_0,q^{-1}(v_0-v)\big)\,\frac{dq_0  }{|q_0|_V}dv_0dv\\
=|q|_V\int (\CF_V\vf)(q_0,\xi) \,(\CF_V\vf')(q^{-1}q_0,-{q^{-1}}^\flat\xi)\,\frac{dq_0  }{|q_0|_V}.
\end{multline*}
Hence, we get when $f_1$ and $f_2$ belong to $A(G)\cap L^2(G)$:
\begin{multline*}
\tilde f(q,\xi)=
\int \overline\omega\big(\phi^{-1}(\xi_0-{q^{-1}}^\flat\xi'),\phi^{-1}(\xi_0-{q^{-1}}^\flat\xi')^{-1}\phi^{-1}(\xi_0-{q^{-1}}^\flat\xi)\big)\\
(\CF_Vf_1)(q,\xi')\, (\CF_Vf_2)\big(\phi^{-1}(q^\flat\xi_0-\xi'),\xi-\xi'\big)d\xi',
\end{multline*}
and Lemma \ref{LSPF1} gives the equality \eqref{SES}. This equality also shows that
$A(G)\cap L^2(G)$ closes to an algebra for the product $\star_\omega$. Indeed, by construction the LHS of \eqref{SES} belongs to $L^2(G)$
and the RHS belongs to $A(G)$.
\ep

We close this section by clarifying a question left open in  \cite{BGNT3}. The point is that there is an
alternative  way  to define a dual $2$-cocycle directly 
from the product $\star_\omega$.
 Namely, if we express  $\star_\omega$ as
$$
f_1\star_\omega f_2(g)=\int_{G\times G} k_\omega(g_1,g_2)\, f_1(gg_1)\,f_2(gg_2)\,dg_1 \,dg_2 ,
$$
for a certain (Bruhat) distribution $k_\omega$ on $G\times G$, then one can
define (the adjoint of) a dual cocycle $\tilde\Omega_\omega$ by the formula:
\begin{align}
\label{Alter}
\tilde\Omega_\omega^*=\int_{G\times G} k_\omega(g_1,g_2)\,\lambda_{g_1}\otimes \lambda_{g_2}\,dg_1\, dg_2 .
\end{align}
This follows because  the associativity of $\star_\omega$ is, at least formally, equivalent
to the dual cocycle equation \eqref{eq:dual-cocycle} for $\tilde \Omega_\omega$. Of course such a formula needs to be precisely defined and unitarity
has to be proven. In general,   this can be  a very difficult task. 
Observe however that Theorem  \ref{OmegaF} and \cite[Lemma 3.10]{BGNT3} already show that the dual $2$-cocycles \eqref{eq:dual-cocycle} and \eqref{Alter}
are equal when $\omega=1$.

We are going to prove that this equality still holds for any $2$-cocycle $\omega\in Z^2(Q,\mathbb T)$.
First we need to explain how  \eqref{Alter} is precisely defined.
Starting with the formula given in Lemma \ref{LSPF1} and undoing the Fourier transform without paying attention to convergence issues, we get
$$
f_1\star_\omega f_2(g_0)=\int e^{i\langle q^\flat\xi_0-\xi_0,v\rangle}\,
(\CF_V^*\tilde\omega_q)(w)\, f_1(g_0(1,v))\,f_2(g_0(q,w))
\,\frac{dq}{|q|_V}\,dv\, dw,
$$
where we have defined
\begin{align}
\label{to}
\tilde\omega_q(\xi):=\overline\omega(q,\phi^{-1}({q^{-1}}^\flat\xi+\xi_0)).
\end{align}
Note that when $\omega=1$, the Fourier transform gives a $\delta$-factor and we obtain  the product $\star$ considered in \cite{BGNT3}.
Hence, we are left to give a meaning to the formal left convolution operator:
\begin{align*}
\tilde\Omega_\omega^*&=\int _{G\times V} e^{i\langle q^\flat\xi_0-\xi_0,v\rangle}\,
(\CF_V^*\tilde\omega_q)(w)\, \lambda_{(1,v)}\otimes\lambda_{(q,w)}
\,\frac{dq}{|q|_V}\,dv\,dw.
\end{align*}
Passing to  the adjoint, ignoring the order of integrations and using
$$
\int _ V 
(\overline{\CF_V^*\tilde\omega_q})(w)\lambda_{(q,w)^{-1}}dw
=\lambda_{(q,0)^{-1}}\int _ V (\CF_V^*\overline{\tilde\omega_q})(w)\lambda_{(1,w)}dw
=\lambda_{(q,0)^{-1}}\,\CF_{V}^*\,\tilde\omega_q^*\,\CF_{V},
$$
we can now  define the dual $2$-cocycle \eqref{Alter} as the sesquilinear form on
the algebraic tensor product $C_c(G)\otimes\CF C_c(G)$ (here $\CF C_c(G)$ denotes
the dense subspace of $L^2(G)$ consisting in functions $f$ such that $\CF_Vf\in C_c(Q\times \hat V)$) given by:
\begin{align}
\label{Omega:def}
\tilde\Omega_\omega[f_1 ,f_2 ]:=\int_G e^{-i\langle q^\flat\xi_0-\xi_0,v\rangle}\,\big\langle {f_1 },
\big(\lambda_{(1,v)^{-1}}\otimes\lambda_{(q,0)^{-1}}\big)\big( 1\otimes\CF_{V}^*\,\tilde\omega_q^*
\CF_{V}\big)f_2  \big\rangle\,
\frac{dqdv}{|q|_V}. 
\end{align}

\begin{proposition}
We have for $f_1 ,f_2 \in C_c(G)\otimes \CF C_c(G)$:
$$
\tilde\Omega_\omega[f_1 ,f_2 ]=\langle f_1 ,\Omega_\omega f_2 \rangle.
$$
\end{proposition}
\bp
Let $f_1 ,f_2 \in C_c(G)\otimes \CF C_c(G)$.
We have
\begin{align*}
&\tilde\Omega_\omega[f_1 ,f_2 ]=
\int e^{-i\langle q^\flat\xi_0-\xi_0,v\rangle}\,\big\langle ( 1\otimes\CF_{V})\big(1\otimes\lambda_{(q,0)}\big)f_1 ,
\big(\lambda_{(1,v)^{-1}}\otimes1\big)
(1\otimes\tilde\omega_q^* \CF_{V}\big)f_2  \big\rangle\,
\frac{dq}{|q|_V}\,dv.
\end{align*}
Using then \eqref{CFL} and remembering \eqref{to}, we  get
\begin{multline*}
\tilde\Omega_\omega[f_1 ,f_2 ]=\int e^{-i\langle q^\flat\xi_0-\xi_0,v\rangle}\,
\overline{( 1\otimes\CF_{V})f_1 }(q_1,v_1;q_2,{q^{-1}}^\flat\xi_2) \,\omega(q,\phi^{-1}({q^{-1}}^\flat\xi_2+\xi_0))
\\
(1\otimes\CF_{V})f_2 
(q_1,v_1+v;qq_2,\xi_2)
\frac{dq  dv }{|q|_V}\frac{dq_1  dv_1 }{|q_1|_V}\frac{dq_2 d\xi _2 }{|q_2|_V}.
\end{multline*}
Since $( 1\otimes\CF_{V})f_j\in C_c(G\times Q\times\hat V)$, $j=1,2$, we can  use the Theorem of Fubini 
to rewrite the integrals  over the variables $v$ and $v_1$ as Fourier transforms:

\begin{multline*}
\tilde\Omega_\omega[f_1 ,f_2 ]=\int \overline{( \CF_V\otimes\CF_V)f_1 }(q_1,q^\flat\xi_0-\xi_0;q_2,{q^{-1}}^\flat\xi_2)\,
\omega(q,\phi^{-1}({q^{-1}}^\flat\xi_2+\xi_0))\\
(\CF_V\otimes\CF_V)f_2 
(q_1, q^\flat\xi_0-\xi_0 ;qq_2,\xi_2)
\frac{dq }{|q|_V}\frac{dq_1 }{|q_1|_V}\frac{dq_2 d\xi _2 }{|q_2|_V}.
\end{multline*}
Setting $\xi_1:=q^\flat\xi_0-\xi_0$, we arrive at:
\begin{multline*}
\tilde\Omega_\omega[f_1 ,f_2 ]
=\int 
\omega\big(\phi^{-1}(\xi_0+\xi_1),\phi^{-1}(\xi_0+\xi_1)^{-1}\phi^{-1}(\xi_0+\xi_1+\xi_2)\big)\\
\overline{( \CF_V\otimes\CF_V)f_1 }(q_1,\xi_1;q_2,{\phi^{-1}(\xi_0+\xi_1)^{-1}}^\flat\xi_2)
(\CF_V\otimes\CF_V)f_2 
(q_1, \xi_1 ;\phi^{-1}(\xi_0+\xi_1)q_2,\xi_2)
\frac{dq_1  d\xi _1 }{|q_1|_V}\frac{dq_2 d\xi _2 }{|q_2|_V}\\
=\int |\phi^{-1}(\xi_0+\xi_1)|_V^{-1}
\omega\big(\phi^{-1}(\xi_0+\xi_1),\phi^{-1}(\xi_0+\xi_2)\big)
\overline{( \CF_V\otimes\CF_V)f_1 }(q_1,\xi_1;q_2,\xi_2)\\
(\CF_V\otimes\CF_V)f_2 
(q_1, \xi_1 ;\phi^{-1}(\xi_0+\xi_1)q_2,\phi^{-1}(\xi_0+\xi_1)^\flat\xi_2)
\frac{dq_1  d\xi _1 }{|q_1|_V}\frac{dq_2 d\xi _2 }{|q_2|_V},
\end{multline*}
and the proof follows by Theorem  \ref{OmegaF}.
\ep

\subsection{The multiplicative unitary}

By \cite[Proposition 5.4]{DC}, the multiplicative unitary of the locally compact quantum group 
$(W^*(G),\hat\Delta_{\Omega_\omega})$ is given by the formula:
$$
\hat W_{\Omega_\omega}
=(\tilde J\otimes\hat J)\,\Omega_\omega\,\hat W^*\,(J\otimes\hat J)\,\Omega_\omega^*.
$$
As already said,  the missing ingredient to compute it is  $\tilde J$ the modular involution of  
the canonical weight    \eqref{TVF} of the $G$-Galois object $(W^*(\hat G, \Omega_\omega),\Ad\rho)$,
with respect to the GNS map  \eqref{TL}.  To determine $\tilde J$, it is easier to work with the isomorphic (by \cite[Proposition 2.9]{BGNT3})
$G$-Galois object  $(B(L^2(Q)),\Ad \pi_\omega)$. The first task is to determine explicitly this isomorphism, at least in terms of the generators
$\pi_{\Omega_\omega}(f)$ and $\Op_\omega(f)$.
In order to do that, we will consider the third equivalent $G$-Galois object $(L^{\star_\omega}(L^2(G))'',\Ad\lambda)$. Here,
  $L^{\star_\omega}$  is the left regular representation of the algebra $(L^2(G),\star_\omega)$, i$.$e$.$ it is defined by
$$
L^{\star_\omega}(f_1)f_2=f_1\star_\omega f_2.
$$
By conjugation by $\Op$, the representation $L^{\star_\omega}$ corresponds to the left regular representation of $\HS(\CH)$ on itself.
Since this representation is a multiple of the standard representation of  $\HS(\CH)$ on $\CH$, there exists a unique isomorphism
$B(\CH)\cong L^{\star_\omega}(L^2(G))''$ which, for $f\in L^2(G)$, maps  $\Op_\omega(f)$ to  $L^{\star_\omega}(f)$ and
which intertwines $\Ad\pi_\omega$ and $\Ad\lambda$.

\begin{definition}
Denote by $\CL_0(Q\times \hat V)$ the union, over the pairs $(K,L)$ of compact subsets of  $Q$, of the spaces $\CL_{K,L}(Q\times \hat V)$ consisting in
essentially bounded functions on $Q\times \hat V$ essentially supported on the compact set
$$
\{(q,\xi)\in Q\times\hat V| q\in K,\;\xi_0-{q^{-1}}^\flat\xi\in \phi(L)\}.
$$
Let then  $\CF\CL_0(G)$ be the subspace of $L^2(G)$ consisting in functions whose Fourier transforms belong to $\CL_0(Q\times \hat V)$.
\end{definition}

\begin{definition}
Let $T$ be the unbounded operator on $L^2(G)$  given on the domain $\CF\CL_0(G)$  by:
\begin{align}
\label{T}
\CF_V T f(q,\xi)=\frac{\Delta_Q\big(\phi^{-1}(\xi_0-{q^{-1}}^\flat\xi)\big)^{1/2}}{\big|\phi^{-1}(\xi_0-{q^{-1}}^\flat\xi)\big|_V^{1/2}}\CF_Vf(q,\xi),
\end{align}
where we recall that $\Delta_Q$ denotes the modular function of $Q$.
\end{definition}

\begin{remark}
 It is proven in \cite[Lemma 3.16]{BGNT3} that $\CF\CL_0(G)$  is dense in $L^2(G)$, that $\CF\CL_0(G)\cap A(G)$  is dense in $A(G)$,
 that $T$ stabilizes $\CF\CL_0(G)$ and that $T(\CF\CL_0(G)\cap A(G))$ is dense in $L^2(G)$.

 Clearly, the operator $T$ is affiliated to the commutative von Neumann subalgebra of $B(L^2(Q))$ given by $\CF_V^* L^\infty (Q\times\hat V)\CF_V$.
 Note also that the unitary operator
 $\mathcal V_\omega$  defined  in Lemma \ref{Vomega}, belongs to the same commutative von Neumann algebra and stabilizes $\CF\CL_0(G)$. Therefore,
 $T$ and $\mathcal V_\omega$ commute on the domain $\CF\CL_0(G)$.
\end{remark}
 \begin{proposition}
 We have $W^*(\hat G, \Omega_\omega)=\CJ\,L^{\star_\omega}(L^2(G))''\,\CJ$.
 More precisely, we have:
 $$
 \pi_{\Omega_\omega}(f)=\CJ\, L^{\star_\omega}(Tf)\,\CJ\quad\mbox{for all}\quad f\in \CF\CL_0(G)\cap A(G).
 $$
 \end{proposition}
 \bp
 First we claim that the operator  of multiplication by the modular function $\Delta_G$ on $L^2(G)$ (given with maximal domain) is affiliated with $L^{\star_\omega}(L^2(G))''$.
 Indeed, with $\mathcal V_\omega$ the unitary operator defined  in Lemma \ref{Vomega}, we have
 $$
 L^{\star_\omega}(f)= \mathcal V_\omega^*\,L^{\star}(\mathcal V_\omega f)\,\mathcal V_\omega, \quad\forall f\in L^2(G),
 $$
 and therefore  $L^{\star_\omega}(L^2(G))''=\mathcal V_\omega^*\,L^{\star}(L^2(G))''\,\mathcal V_\omega$. The claim follows then
 from the fact that $\Delta_G$  is affiliated with $L^{\star}(L^2(G))''$ (\cite[Proposition 3.17]{BGNT3}), that $\mathcal  V_\omega$ preserves the
 maximal domain of   $ \Delta_G$ and that $\mathcal  V_\omega$ and  $\Delta_G$ commute on that domain.

Then, we claim that for all $f\in \CF\CL_0(G)$, we have $\Delta_G^{-1/2}L^{\star_\omega}(f)\Delta_G^{1/2}=L^{\star_\omega}(Tf)$ on the domain
$L^2(G)\cap \Delta_G^{-1/2}L^2(G)$.
Indeed, by \cite[Proposition 3.17]{BGNT3}, we have $\Delta_G^{-1/2}L^{\star}(f)\Delta_G^{1/2}=L^{\star}(Tf)$ on
that domain. Then, since $\mathcal  V_\omega$ preserves $L^2(G)\cap \Delta_G^{-1/2}L^2(G)$,
 we deduce the following equalities on the domain $L^2(G)\cap \Delta_G^{-1/2}L^2(G)$:
\begin{align*}
\Delta_G^{-1/2}L^{\star_\omega}(f)\Delta_G^{1/2}=\Delta_G^{-1/2}\mathcal  V_\omega^*L^{\star}(\mathcal  V_\omega f)\mathcal  V_\omega\Delta_G^{1/2}&=
\mathcal  V_\omega^*\Delta_G^{-1/2}L^{\star}(\mathcal  V_\omega f)\Delta_G^{1/2}\mathcal  V_\omega\\
&=\mathcal  V_\omega^* L^{\star}(T\mathcal  V_\omega f) \mathcal  V_\omega=
\mathcal  V_\omega^* L^{\star}(\mathcal  V_\omega T f) \mathcal  V_\omega=L^{\star_\omega}(Tf).
\end{align*}

Let $\star_{\Omega_\omega}$ be the associative product on $A(G)$ given by the RHS of \eqref{SES}.
We know by \cite[Section 4.1]{NT} that $\pi_{\Omega_\omega}(f_1)\check f_2=(f_1\star_{\Omega_\omega}f_2\check)$ for all $f_1\in A(G)$ and
$f_2\in A(G)\cap C_c(G)$. Hence we deduce from Proposition \ref{ESP} that
for $f\in A(G)\cap \CF\CL_0(G)$ we have $\pi_{\Omega_\omega}(f)=\CJ\Delta_G^{-1/2}L^{\star_\omega}(f)\Delta_G^{1/2}\CJ$
  on $A(G)\cap C_c(G)$, hence on $L^2(G)$  by density since both sides of the equality define bounded operators.
As $A(G)\cap C_c(G)\subset L^2(G)\cap \Delta_G^{-1/2}L^2(G)$, we deduce from what precedes that $\pi_{\Omega_\omega}(f)=\CJ\, L^{\star_\omega}(Tf)\,\CJ$,
at least for $f\in A(G)\cap  \CF\CL_0(G)$.

The first statement follows by density of  $\CF\CL_0(G)\cap A(G)$   in $A(G)$ and  of $T(\CF\CL_0(G)\cap A(G))$ in $L^2(G)$.
 \ep

 \begin{corollary}
 Under the isomorphism
 $(W^*(\hat G, \Omega_\omega),\Ad\rho)\cong(B(L^2(Q)),\Ad \pi_\omega)$,
  the operator $\pi_{\Omega_\omega}(f)$, for $f\in A(G)\cap  \CF\CL_0(G)$,  is mapped  to $\Op_\omega(Tf)$ where $T$ is given in \eqref{T}.
 \end{corollary}

\begin{proposition}
\label{MW}
The modular involution of the canonical weight $\vf$ of the $G$-Galois object  $(W^*(\hat G, \Omega_\omega),\Ad\rho)$ with respect to the GNS map
\eqref{TL} is given by:
$$
\tilde J=\CJ \U_\omega \CJ J,
$$
where $\U_\omega$ is the unitary operator  defined  in Lemma \ref{OU}.
\end{proposition}
\bp
The proof is  identical to the one of   \cite[Proposition 3.24]{BGNT3}.
\ep

Let  $\hat  W_\Omega$ be the multiplicative unitary of the locally compact quantum group  $(W^*(G),\Omega\hat\Delta(.)\Omega^*)$ associated with the dual cocycle
$\Omega$ underlying the genuine representation $\pi$.
 By \cite[Theorem 3.26]{BGNT3} we have for $f\in L^2(Q\times\hat V\times Q\times\hat V)$ (with the measure $|q_1|^{-1}dq_1d\xi_1 |q_2|^{-1}dq_2d\xi_2 $):
$$
(\CF_V\otimes\CF_V)\hat W_{\Omega} (\CF_V^*\otimes\CF_V^*)f=d_w^{1/2} \, (f\circ w),
$$
where $w:(Q\times\hat V)^2\to (Q\times\hat V)^2$ is the pentagonal transformation given in \eqref{PTO}
and $d_w$ is its  Radon-Nikodym derivative.
By \cite[Theorem 4.1]{BGNT3}, the pentagonal transformation $w$ comes from the matched pair
$G_1=Q$, $G_2=(1,\xi_0) Q(1,\xi_0)^{-1}$ of subgroups of the dual crossed product $Q\ltimes\hat V$.

We are ready to formulate the main result of this paper.

\begin{theorem}
For $\omega\in Z^2(Q, \T)$, let $\Theta_\omega: (Q\times\hat V)\times (Q\times\hat V)\to\T$ be
the measurable function given a.e$.$ by
\begin{multline}
\label{Teom}
\Theta_\omega(q_1,\xi_1;q_2,\xi_2):=
\omega\big(\phi^{-1}(\xi_0+\xi_1),\phi^{-1}(\xi_0+\xi_1)^{-1}\phi^{-1}(\xi_0+{q_2^{-1}}^\flat\xi_2)\big)
\\ \overline\omega\big(\phi^{-1}(\xi_0+q_2^\flat\xi_1),\phi^{-1}(\xi_0+q_2^\flat\xi_1)^{-1}\phi^{-1}(\xi_0+\xi_2)\big).
\end{multline}
The multiplicative unitary $\hat W_{\Omega_\omega}$ of the locally compact quantum group
$(W^*(G),\Omega_\omega\hat\Delta(.)\Omega_\omega^*)$ is given by
$$
\hat W_{\Omega_\omega}=(\CF_V^*\otimes\CF_V^*) \Theta_\omega (\CF_V\otimes\CF_V)\hat W_{\Omega}.
$$
Consequently,  $(W^*(G),\Omega_\omega\hat\Delta(.)\Omega_\omega^*)$
is isomorphic to the cocycle bicrossed product quantum group defined
by the matched pair $G_1=Q$, $G_2=(1,\xi_0) Q(1,\xi_0)^{-1}$ of subgroups of $Q\ltimes\hat V$ and by the pentagonal $2$-cocycle $\Theta_\omega$.
Moreover, the map $\omega\mapsto \Theta_\omega$ induces a group homomorphism from measurable cohomology $H^2(Q,\T)$
to  pentagonal cohomology  $H^2(w,\T)$.
\end{theorem}
\bp
From Proposition \ref{MW}, we have  the formula:
$$
\hat W_{\Omega_\omega}
=(\CJ \U_\omega J\CJ\otimes\hat J)\,\Omega_\omega\,\hat W^*\,(J\otimes\hat J)\,\Omega_\omega^*.
$$
The explicit expression for $(\CF_V^*\otimes\CF_V^*)\hat W_{\Omega_\omega} (\CF_V\otimes\CF_V)$ follows then by
an easy generalization of the computations given in the proof of \cite[Theorem 3.26]{BGNT3}; one just has to keep track of the three extra $\omega$-factors.
More precisely,  Theorem \eqref{OmegaF} (and its proof) gives the formulas for $(\CF_V\otimes\CF_V)\Omega_\omega (\CF_V^*\otimes\CF_V^*)$ and  for
$(\CF_V\otimes\CF_V)\Omega_\omega^* (\CF_V^*\otimes\CF_V^*)$. We also have (see the proof of \cite[Theorem 3.26]{BGNT3}):
\begin{align*}
(\CF_V J\CF_V^*\,f)(q,\xi)=\overline{f(q,-\xi)},\qquad
(\CF_V \hat J\CF_V^*\,f)(q,\xi)=\frac{|q|^{3/2}}{\Delta_Q(q)^{1/2}}\overline{f\big(q^{-1},{q^{-1}}^\flat\xi\big)},
\end{align*}
and
\begin{equation*}
\big((\CF_V\otimes \CF_V)\hat W^* (\CF_V^*\otimes \CF_V^*)f\big)(q_1,\xi_1;q_2,\xi_2)
= |q_2| f\big(q_2^{-1}q_1,{q_2^{-1}}^\flat\xi_1;q_2,\xi_1+\xi_2\big).
\end{equation*}
Moreover, we deduce from Lemma \ref{OU} that we have  a.e$.$
\begin{multline*}
(\CF_V\CJ\CU_\omega\CJ\CF_V^*f)(q,\xi)
=\frac{ |\phi^{-1}(\xi_0+\xi)|^{3/2}}{ \Delta_Q\big(\phi^{-1}(\xi_0+\xi)\big)^{1/2}}\omega(\phi^{-1}(\xi_0+\xi),\phi^{-1}(\xi_0+\xi)^{-1})\\
f\big(\phi^{-1}(\xi_0+\xi)^{-1}q,{\phi^{-1}(\xi_0+\xi)^{-1}}^\flat\xi\big).
\end{multline*}
From all that, we get a.e$.$
\begin{multline*}
\big((\CF_V\otimes\CF_V)(\CJ \U_\omega \CJ\otimes1)(J\otimes\hat J)\,\Omega_\omega\,\hat W^*\,(J\otimes\hat J)\,
\Omega_\omega^* (\CF_V^*\otimes\CF_V^*)f\big)(q_1,\xi_1;q_2,\xi_2)=\\
\omega(\phi^{-1}(\xi_0+\xi_1),\phi^{-1}(\xi_0+\xi_1)^{-1})
 \overline{\omega}\big(\phi^{-1}(\xi_0+\xi_1)^{-1},\phi^{-1}(\xi_0+{q_2^{-1}}^\flat \xi_2)\big)\\
 \overline{\omega}\big(\phi^{-1}(\xi_0+q_2^\flat\xi_1),\phi^{-1}(\xi_0+q_2^\flat\xi_1)^{-1}\phi^{-1}(\xi_0+\xi_2)\big) 
 |\phi^{-1}({q_2^{-1}}^\flat\xi_0+\xi_1)|\\
 f\big(q_2q_1,q_2^\flat\xi_1;\phi^{-1}({q_2^{-1}}^\flat\xi_0+\xi_1)^{-1}\phi^{-1}(\xi_0+\xi_1),{\phi^{-1}({q_2^{-1}}^\flat \xi_0+\xi_1)^{-1}}^\flat
({q_2^{-1}}^\flat\xi_2-\xi_1)\big).
\end{multline*}
The cocycle identity for $(\phi^{-1}(\xi_0+\xi_1),\phi^{-1}(\xi_0+\xi_1)^{-1},\phi^{-1}(\xi_0+{q_2^{-1}}^\flat \xi_2))$ (plus our normalization) gives
\begin{multline*}
\omega\big(\phi^{-1}(\xi_0+\xi_1),\phi^{-1}(\xi_0+\xi_1)^{-1}\phi^{-1}(\xi_0+{q_2^{-1}}^\flat\xi_2)\big)
{\omega}\big(\phi^{-1}(\xi_0+\xi_1)^{-1},\phi^{-1}(\xi_0+{q_2^{-1}}^\flat \xi_2)\big)\\=
\omega(\phi^{-1}(\xi_0+\xi_1),\phi^{-1}(\xi_0+\xi_1)^{-1}),
\end{multline*}
from which the formula follows.

Since $\hat  W_\Omega$ and $\hat W_{\Omega_\omega}$ are multiplicative unitaries,
we deduce that $\Theta_\omega$ defines a class in the pentagonal cohomology group $H^2(w,\T)$. Thus
 $(W^*(G),\Omega_\omega\hat\Delta(.)\Omega_\omega^*)$
is isomorphic to a cocycle bicrossed product quantum group.

Clearly the map $\omega\mapsto\Theta_\omega$ is a group homomorphism. It remains to show that this map induces a map 
in cohomology  $H^2(Q,\T)\to H^2(w,\T)$.
So assume that $\omega$ is trivial, that is $\omega=\partial u=\frac{u\otimes u}{\Delta u}$, where $u:Q\to\T$ is measurable.
We need to show that $\Theta_{\partial u}$ is
trivial as well, i.e$.$ that  there exists a measurable function $a:Q\times\hat V\to\T$, such that $\Theta_{\partial u}=\frac{a\otimes a}{( a\otimes a)\circ w}$.
A computation shows that:
\begin{multline*}
\Theta_{\partial u}(q_1,\xi_1;q_2,\xi_2)=
\frac{u(\phi^{-1}(\xi_0+\xi_1))u(\phi^{-1}(\xi_0+\xi_1)^{-1}\phi^{-1}(\xi_0+{q_2^{-1}}^\flat\xi_2))}
{u(\phi^{-1}(\xi_0+{q_2^{-1}}^\flat\xi_2))}\\
\frac{u(\phi^{-1}(\xi_0+\xi_2))}{u(\phi^{-1}(\xi_0+q_2^\flat\xi_1))u(\phi^{-1}(\xi_0+q_2^\flat\xi_1)^{-1}\phi^{-1}(\xi_0+\xi_2))}.
\end{multline*}
Now, consider the function $a:Q\times \hat V\to\T$  defined a.e$.$ by
$$
a(q,\xi):=u\big(\phi^{-1}(\xi_0+\xi)\big)
\ \overline u\big(\phi^{-1}(\xi_0+{q^{-1}}^\flat\xi)\big).
$$
Then we have:
\begin{multline*}
\frac{(a\otimes a) }{( a\otimes  a) \circ w_\Omega}(q_1,\xi_1;q_2,\xi_2)=\frac{u\big(\phi^{-1}(\xi_0+\xi_1)\big)}{ u\big(\phi^{-1}(\xi_0+{q_1^{-1}}^\flat\xi_1)\big)}
\frac{u\big(\phi^{-1}(\xi_0+\xi_2)\big)}{ u\big(\phi^{-1}(\xi_0+{q_2^{-1}}^\flat\xi_2)\big)}
\frac{u\big(\phi^{-1}(\xi_0+{q_1^{-1}}^\flat\xi_1)\big)}{ u\big(\phi^{-1}(\xi_0+q_2^\flat\xi_1)\big)}\\
\frac{ u\big(\phi^{-1}(\xi_0+{\phi^{-1}(\xi_0+\xi_1)^{-1}}^\flat
({q_2^{-1}}^\flat\xi_2-\xi_1))\big)}
{u\big(\phi^{-1}(\xi_0+{\phi^{-1}({q_2^{-1}}^\flat\xi_0+\xi_1)^{-1}}^\flat
({q_2^{-1}}^\flat\xi_2-\xi_1))\big)},
\end{multline*}
which after some simplifications gives exactly $\Theta_{\partial u}(q_1,\xi_1;q_2,\xi_2)$.
\ep

We see no reason why the group homomorphism
\begin{equation}
\label{HtoH}
H^2(Q,\T)\to H^2(w,\T),\quad [\omega]\to[\Theta_\omega],
\end{equation}
 should be surjective in general but we believe that it is always injective. 
 However, at the moment we are only able to prove a weaker statement. Namely, under a slightly stronger
assumption on the dual orbit,  the restriction of this map 
 to the continuous cohomology group $H^2_c(Q,\T)$ is indeed injective:
\begin{proposition}
Assume that the  map $\phi:Q\to \hat V$, $q\mapsto q^\flat\xi_0$ is open.
If $\omega:Q\times Q\to\mathbb \T$ is a  continuous $2$-cocycle such that $[\Theta_\omega]=[1]\in H^2(w,\T)$,
then $[\omega]=[1]\in H^2(Q,\T)$.
\end{proposition}
\bp
Denote by $\mathcal O:= \phi(Q)$  the dual orbit. By assumption, $\mathcal O$ is an open set and  $\phi :Q\to\mathcal O$ is a homeomorphism. 
Consider the continuous map
$$
F:\hat V\times Q\times\hat V\to \hat V\times\hat V\times\hat V\times\hat V,\quad
(\xi_1,q,\xi_2)\mapsto (\xi_0+\xi_1,\xi_0+q^\flat\xi_1,\xi_0+\xi_2,\xi_0+{q^{-1}}^\flat\xi_2),
$$
and define the  open  set:
$$
U:=Q\times F^{-1}(\mathcal O\times\mathcal O\times\mathcal O\times\mathcal O)\subset Q\times\hat V\times Q\times\hat V.
$$
It is easy to see that the complement $U^c$ has  measure zero.
From  continuity of the group cocycle $\omega$ we deduce  that the pentagonal cocycle $\Theta_\omega$ given in \eqref{Teom}
is continuous on $U$.
Moreover,  considering the continuous map
$$
H:\hat V\times Q\to\hat V\times\hat V,\quad
(\xi,q)\mapsto (\xi_0+\xi,\xi_0+q^\flat\xi),
$$
we see the pentagonal transformation $w$  given in \eqref{PTO} is continuous on the open dense set
$$
\tilde{U}:=  Q\times H^{-1}(\mathcal O\times\mathcal O)\times \hat V\supset U.
$$

Let $a:Q\times\hat V\to\T$ be a measurable function such that  $\Theta_\omega=(a\otimes a) ((\overline a\otimes\overline a)\circ w)$. Of course, 
this is an
equality of continuous functions at least on  $U$. Explicitly, we have
\begin{multline*}
\frac{a\otimes a}{(\overline a\otimes\overline a)\circ w}(q_1,\xi_1;q_2,\xi_2)=\\
\frac{a(q_1,\xi_1)\,a(q_2,\xi_2)}
{a\big(q_2q_1,q_2^\flat\xi_1\big)\,a\big(\phi^{-1}({q_2^{-1}}^\flat\xi_0+\xi_1)^{-1}\phi^{-1}(\xi_0+\xi_1),{\phi^{-1}({q_2^{-1}}^\flat\xi_0+\xi_1)^{-1}}^\flat
({q_2^{-1}}^\flat\xi_2-\xi_1)\big)}.
\end{multline*}
For $q_1\in Q$ and $\xi_1,\xi_2\in\mathcal O-\xi_0$, we have $(q_1,\xi_1;e,\xi_2)\in U$ and $\Theta_\omega(q_1,\xi_1;e,\xi_2)=1$. Therefore
$$
a(e,\xi_2)=
a\big(e,{\phi^{-1}(\xi_0+\xi_1)^{-1}}^\flat
(\xi_2-\xi_1)\big),
$$
which entails that $a(e,\xi)=a(e,0)$ for all $\xi\in\mathcal O-\xi_0$.
Of course, we can assume further that $a(e,0)=1$. 
Since $\Theta_\omega$ does not depend on the variable $q_1$, we deduce that there exists a measurable function
$c:Q\times\hat V\to\T$ such that
$$
a(q_2q_1,q_2^\flat\xi_1)=c(q_2,\xi_1)\,a(q_1,\xi_1).
$$
Putting $q_1=e$, we obtain
$$
c(q,\xi)=\overline a(e,\xi)\,a(q,q^\flat\xi)=a(q,q^\flat\xi).
$$
This shows  that the measurable function $a:Q\times\hat V\to \T$ satisfies the  functional equation:
\begin{align}
\label{E}
a(q_2q_1,q_2^\flat\xi_1)=a(q_2,q_2^\flat\xi_1)\,a(q_1,\xi_1).
\end{align}
Now,  the (measurable) solutions of the above equation are all of the form
\begin{align}
\label{k}
a(q,\xi)=k(\xi)\overline{k}({q^{-1}}^\flat\xi),
\end{align}
where $k:\hat V\to\T$ is measurable and, of course, we  may assume  $k(0)=1$.
That  a function of the form \eqref{k} solves \eqref{E} is obvious.
Reciprocally, if $a$ solves \eqref{E}, setting $k(\xi):=\overline a(\phi^{-1}(\xi)^{-1},\xi_0)$, for $\xi\ne 0$ and $k(0)=1$, we get
$$
\frac{k(\xi)}{k({q^{-1}}^\flat\xi)}=\frac{a(\phi^{-1}({q^{-1}}^\flat\xi)^{-1},\xi_0)}{a(\phi^{-1}(\xi)^{-1},\xi_0)}=
\frac{a(\phi^{-1}(\xi)^{-1}q,{\phi^{-1}(\xi)^{-1}}^\flat\xi)}{a(\phi^{-1}(\xi)^{-1},{\phi^{-1}(\xi)^{-1}}^\flat\xi)}=a(q,\xi),
$$
where the last equality follows from \eqref{E}.
Therefore, we have
$$
\frac{a\otimes a}{(\overline a\otimes\overline a)\circ w}(q_1,\xi_1;q_2,\xi_2)=
\frac{k(\xi_1)\,k(\xi_2) \,k\big({\phi^{-1}(\xi_0+\xi_1)^{-1}}^\flat
({q_2^{-1}}^\flat\xi_2-\xi_1)\big)}{k(q_2^\flat\xi_1)\,k\big({\phi^{-1}({q_2^{-1}}^\flat\xi_0+\xi_1)^{-1}}^\flat
({q_2^{-1}}^\flat\xi_2-\xi_1)\big)\,k({q_2^{-1}}^\flat\xi_2)}.
$$

Now, fix $q_1\in Q$, $\xi_1,\xi_2\in\mathcal O-\xi_0$ and let
 $(q_n)_{n\in\mathbb N}$ be a sequence in $Q$ such that $q_n^\flat\to 0$ in ${\rm End}(\hat V)$ pointwise.
Since $q_n^\flat \xi_j\to 0\in\mathcal O-\xi_0$, $j=1,2$, we deduce that $(q_1,q_n^\flat \xi_1,q_n,q_n^\flat \xi_2)$ will
eventually belong to $U$.
Moreover, note that
\begin{multline*}
\Theta_\omega(q_1,\xi_1;q_n,q_n^\flat\xi_2)=
\omega\big(\phi^{-1}(\xi_0+\xi_1),\phi^{-1}(\xi_0+\xi_1)^{-1}\phi^{-1}(\xi_0+\xi_2)\big)\\
\overline\omega\big(\phi^{-1}(\xi_0+q_n^\flat\xi_1),\phi^{-1}(\xi_0+q_n^\flat\xi_1)^{-1}\phi^{-1}(\xi_0+q_n^\flat\xi_2)\big),
\end{multline*}
 converges to 
$$
\omega\big(\phi^{-1}(\xi_0+\xi_1),\phi^{-1}(\xi_0+\xi_1)^{-1}\phi^{-1}(\xi_0+\xi_2)\big).
$$
On the other hand, we have
$$
\frac{a\otimes a}{(\overline a\otimes\overline a)\circ w}(q_1,\xi_1;q_n,q_n^\flat\xi_2)=
\frac{k(\xi_1)\,k\big({\phi^{-1}(\xi_0+\xi_1)^{-1}}^\flat
(\xi_2-\xi_1)\big)\,k(q_n^\flat\xi_2)}
{k(\xi_2)\,k(q_n^\flat\xi_1)\,k\big({\phi^{-1}(\xi_0+q_n^\flat\xi_1)^{-1}}^\flat
(q_n^\flat\xi_2-q_n^\flat\xi_1)\big)}.
$$
which converges to
$$
\frac{k(\xi_1) \,k\big({\phi^{-1}(\xi_0+\xi_1)^{-1}}^\flat
(\xi_2-\xi_1)\big)}{k(\xi_2)}.
$$
Setting now $q_j:=\phi^{-1}(\xi_0+\xi_j)^{-1}$ and $u(q):=k({q^{-1}}^\flat\xi_0-\xi_0)$, we get
$$
\omega(q_1,q_1^{-1}q_2)=\frac{u(q_1)u(q_1^{-1}q_2) }{u(q_2)},
$$
and therefore $\omega$ represents the trivial class in $H^2(Q,\mathbb T)$.
\ep

\begin{remark}
When $Q$ is Abelian, we have $H^2(Q,\mathbb T)\simeq H^2_c(Q,\mathbb T)$. In this case, the map \eqref{HtoH} is 
injective as soon as $\phi$ is open.
\end{remark}

\end{document}